\documentclass[11pt,reqno]{amsart}
\usepackage{mathrsfs}
\usepackage{hyperref}
\makeatletter
\usepackage{amsfonts}
\usepackage{amsmath}
\usepackage{amssymb}
\usepackage{multicol}
\usepackage{stmaryrd}
\usepackage{cite}
\usepackage{epsfig}
\usepackage{color}
\usepackage{graphics}
\usepackage{graphicx}
\usepackage{epstopdf}

\allowdisplaybreaks

\numberwithin{equation}{section}

\newtheorem{theorem}{Theorem}[section]
\newtheorem{lemma}[theorem]{Lemma}

\newtheorem{proposition}[theorem]{Proposition}

\usepackage{xcolor}

\newcommand{\wtd}{\widetilde}

\newcommand{\R}{\mathbb{R}}

\newcommand{\cuad}{{\sqcap\kern-.68em\sqcup}}

\renewcommand{\phi}{\varphi}
\newcommand{\cA}{{\mathcal A}}

\newcommand{\cC}{{\mathcal C}}

\newcommand{\cG}{{\mathcal G}}

\newcommand{\cI}{{\mathcal I}}
\newcommand{\cJ}{{\mathcal J}}

\newcommand{\bX}{{\mathbb X}}

 \newcommand{\bT}{{\mathbb T}}

\newcommand{\loc}{{\rm loc}}

\newcommand{\dd}{\,\mathrm{d}}
\newcommand{\p}{\partial}
\newcommand{\abswen}[1]{\left\lvert #1\right\rvert}

\newcommand{\eps}{\varepsilon}

\begin{document}

\title[Propagation rates of the fractional Fisher-KPP equation]{Precise propagation rates of the fractional  Fisher-KPP  equation with free boundaries}
 \author[H. Chen, Y. Du, W. Ni]{Huyuan Chen$^{1,2}$,\quad  Yihong Du$^{3}$\quad  and\quad  Wenjie Ni$^{3}$}
\thanks{\hspace{-.55cm}
\mbox{\  } $1$ Center for Mathematics and Interdisciplinary Sciences, Fudan University, Shanghai 200433, PR China.\\
\mbox{\ }  $2$ Shanghai Institute for Mathematics and Interdisciplinary Sciences, Shanghai 200433, PR China\\
\mbox{\ \ \ \ } Emails: chenhuyuan@simis.cn,\ chenhuyuan@yeah.net.}
\thanks{\hspace{-.55cm}
\mbox{\  } $3$ School of Science and Technology, University of New England, Armidale, NSW 2351, Australia.\\
\mbox{\ \ \ \ }  Emails: ydu@une.edu.au (Du), wni2@une.edu.au (Ni).}

\date{\today}

\maketitle


\begin{abstract}
In this paper, we determine the propagation rates of the Fisher-KPP equation with free boundaries and fractional diffusion $(-\Delta)^s$ $(0<s<1)$, which was considered in \cite{CDN26} as a model for species spreading, where the density function $u(t,x)$ satisfies the fractional Fisher-KPP equation $u_t+(-\Delta)^su=f(u)$ over a one-dimensional spatial domain $(g(t), h(t))$, representing the expanding population range, with $x=g(t)$ and $x=h(t)$  the free boundaries. It was shown in \cite{CDN26} that the model is well-posed, and its long-time dynamics exhibits a spreading-vanishing dichotomy. 
We show in this paper that, when spreading occurs, 
 \[
\begin{cases}
\lim_{t\to\infty}\frac{h(t)}{t^{1/(2s)}}
=\lim_{t\to\infty}\frac{-g(t)}{t^{1/(2s)}}
=\left(\frac{2^{1-2s}\mu}{1-2s}\right)^{1/(2s)} & \mbox{ if $s\in(0,1/2)$,}\\[2mm]
\lim_{t\to\infty}\frac{h(t)}{t\ln t}
=\lim_{t\to\infty}\frac{-g(t)}{t\ln t}
=\mu & \mbox{ if $s=1/2$,}\\[2mm]
\lim_{t\to\infty}\frac{h(t)}{t}
=\lim_{t\to\infty}\frac{-g(t)}{t}
=c_0(s) & \mbox{ if $s\in(1/2,1)$},
\end{cases}
\]
where $c_0(s)\in (0,\infty)$ is uniquely determined by the associated semi-wave problem. The level sets of $u(t,x)$ propagate with the same rates as the free boundaries.
The proof uses delicate upper and lower solutions constructed from a
smooth profile function in the case $s\in (0, 1/2]$, and from the semi-wave profile when $s\in (1/2, 1)$.

\bigskip

\noindent
\textbf{Keywords}: Fractional Laplacian, free boundary, propagation rates.

\medskip

\noindent
\textbf{AMS Subject Classification}:  35B40, 35R11, 35R35

\end{abstract}

\setcounter{equation}{0}
 
 \section{Introduction } 

We determine the precise propagation rates governed by the following  free boundary problem with fractional diffusion:
\begin{equation}\label{eq 1.1}
	\begin{cases}
		\partial_t u +d(-\Delta)^s  u = f(u), & t>0,\ x\in \bigl(g(t),h(t)\bigr),\\[3pt] 
		u(t,x)=0, & t>0,\ x \in \mathbb{R}\setminus \bigl(g(t),h(t)\bigr),\\[3pt] 
		\displaystyle h'(t)= \mu \int_{g(t)}^{h(t)}\int_{h(t)}^{+\infty} \frac{u(t,x)}{|x-y|^{2s+1}}\,dy\,dx, & t>0,\\[3pt] 
		\displaystyle g'(t)=-\mu \int_{g(t)}^{h(t)}\int_{-\infty}^{g(t)} \frac{u(t,x)}{|x-y|^{2s+1}}\,dy\,dx, & t>0,\\[3pt] 
		u(0,x)=u_0(x), & x\in \overline{\mathcal{O}_0} =(-h_0,h_0), \\[3pt] 
		h(0)=-g(0)=h_0,
	\end{cases}
\end{equation}
where $s\in(0,1)$, $\mu,d,h_0$ are positive constants, and $(-\Delta)^s$ denotes the
one-dimensional fractional Laplacian. The initial datum $u_0$ is a nonnegative, nonzero
continuous function on $[-h_0,h_0]$ with $u_0(\pm h_0)=0$, and extended by 0 outside $[-h_0, h_0]$. 

Throughout the paper we assume that  $f$ is a Fisher-KPP type nonlinearity, namely, it satisfies
\begin{description}
\item[(f1)]\ $f\in C^1(\R_+)$, $f(0)=f(1)=0$,
$f'(0)>0>f'(1)$, $f(u)/u$ is non-increasing for $u>0$.
\end{description}
Therefore $f(u)>0$ in $(0,1)$,  $f(u)< 0$ in
$(1,\infty)$, and $f(u)/u\leq f'(0)$ for $u>0$. So the assumption ${\bf (f1)}$ is a
slightly strengthened version of the well-known KPP condition (\cite{KPP}), containing the Fisher nonlinearity $u(1-u)$ (\cite{F}) as a special case.

Problem \eqref{eq 1.1} is a natural extension of the free boundary model  in \cite{CDLL}, where the nonlocal diffusion term is defined via a continuous integrable kernel function $J(x)$. It can also be viewed as a nonlocal version of the free boundary model  in \cite{DL, DLou}, whose local diffusion term is given by the classical Laplacian.

However, the well-posedness of \eqref{eq 1.1}  has been open for quite some time  due to the  difficulties caused by the singularity and non-integrability of the associated kernel function of $(-\Delta)^s$, and the lack of general enough regularity results for operators extending the standard fractional Laplacian. 

These difficulties have been overcome recently in \cite{CDN26} by introducing an approximation approach, based on existing interior regularity estimates of \cite{FR2017} and boundary estimates arising from the heat kernel bounds of \cite{CKS}.  It was shown in \cite{CDN26} that \eqref{eq 1.1} has a unique solution defined for all $t>0$, and its long-time dynamics  exhibits a spreading-vanishing dichotomy, similar to the earlier models of \cite{DL, DLou, CDLL}.  

In view of the results in \cite{DLZ, DN-MA, DN25} on precise propagation profiles of the nonlocal free boundary model of \cite{CDLL}, and similar results in \cite{DLou, DMZ} for the local free boundary model of \cite{DL, DLou},
it is natural to ask, when spreading happens, what are the propagation rates of \eqref{eq 1.1}? In this paper,  we give a rather complete answer to this question.

It is well known that for the corresponding Cauchy problem
$u_t+d(-\Delta)^s u=f(u)$ in $\R$ with compactly supported nonnegative initial data, the
level sets of the solution $u(t,x)$ propagate exponentially fast in time (see 
\cite{CR13}). 

For the corresponding nonlocal free boundary model of \eqref{eq 1.1} with a continuous and integrable dispersal kernel $J(x)$,
the threshold condition for
finite speed spreading is, by 
\cite{DLZ}, 
\[
\int_0^\infty xJ(x)\,\dd x<\infty. 
\]
Precise propagation rates for this model was obtained in \cite{DLZ, DN-MA, DN25}.

Since
\[
(-\Delta)^s u(x)={\rm p.v.} \int_{\R}J_s(x-y)[u(x)-u(y)]\,\dd y,
\]
with
\[
J_s(x):=\frac{c_s}{|x|^{-(1+2s)}},\ \  c_s=2^{2s}\pi^{-1/2}s\,
\frac{\Gamma\!\left(\frac{1+2s}{2}\right)}{\Gamma(1-s)}>0,
\]
formally problem \eqref{eq 1.1} has the same form to the nonlocal model in \cite{CDLL, DLZ, DN-MA, DN25}, except that its dispersal kernel $J_s(x)$ has a singularity at $x=0$ and is not integrable over $\R$.

We will show that the spreading behaviour of \eqref{eq 1.1} is determined sorely by the tail property of $J_s(x)$; its singularity at $x=0$ does not make any visible contributions.
More precisely, the spreading rate of \eqref{eq 1.1} is the same as that for the free boundary model of \cite{CDLL} obtained in \cite{DLZ, DN-MA, DN25}, provided that $J_s(x)$ has the same decay rate  at infinity as the continuous kernel $J(x)$ in \cite{CDLL}.
So, in particular,   the threshold condition for \eqref{eq 1.1} to have a finite spreading speed is
$s>1/2$,
 and for $s\in(0,1/2)$ the spreading has the power rate $t^{1/(2s)}$, corresponding to the case
$\alpha=1+2s\in(1,2)$ in \cite{DN-MA}.
\medskip

Our main results are listed below.

  \begin{theorem} \label{teo 1} 
  
 Suppose that  $s\in(0, 1)$, {\bf (f1)} is satisfied, and  spreading happens to the  solution $(u, g, h)$ of \eqref{eq 1.1}. Then the following conclusions hold.
\begin{itemize}

 \item[(i)] For $s\in(0, \frac12)$, 
\begin{equation}\label{eq:main-result}
\lim_{t\to\infty}\frac{h(t)}{t^{1/(2s)}}
=\lim_{t\to\infty}\frac{-g(t)}{t^{1/(2s)}}
=\left(\frac{2^{1-2s}\mu}{1-2s}\right)^{1/(2s)}.
\end{equation}

 \item[(ii)] For $s=\frac12$, 
\begin{equation}\label{eq:main-result-half}
\lim_{t\to\infty}\frac{h(t)}{t\ln t}
=\lim_{t\to\infty}\frac{-g(t)}{t\ln t}
=\mu.
\end{equation}

\item[(iii)] For $s\in(\frac12, 1)$,   there exists a unique $c_0=c_0(s)>0$ such that
\begin{equation}\label{speed}
\lim_{t\to\infty}\frac{h(t)}{t}=-\lim_{t\to\infty} \frac{g(t)}{t}=c_0.
\end{equation}
\end{itemize}
\end{theorem}
 \bigskip
 
 The positive constant $c_0$ in \eqref{speed} is uniquely determined by the associated semi-wave problem, consisting of the following equations:
\begin{equation}\label{eq 1.1-ell}
 \left\{
\begin{array}{cll}
 d(-\Delta)^s w-c w' = f(w)  \quad
  &{\rm in}\  \, (-\infty, 0)
  \\[2mm]
 0\leq w\leq 1 \quad
  &{\rm in}\  \, \R, 
\\[2mm]
 \displaystyle   w(-\infty)=1,
\ \ \ w=0 \quad  &{\rm in}\  \, [0,+\infty), 
  \end{array}
 \right.
 \end{equation}
 and
 \begin{equation}\label{c0}
 c= \frac{\mu }{2s}  \int_{-\infty}^{0}  w(y)|y|^{-2s} dy.
 \end{equation}
 
 The conclusions on the above semi-wave problem are stated in the following theorem: 
 
  \begin{theorem}\label{teo ell} Suppose that $s\in(\frac12,1)$  and $\bf (f1)$ is satisfied. Then the following conclusions hold:
  \begin{itemize}
  \item[$(a)$]  For each $c\geq 0$, \eqref{eq 1.1-ell} has a   unique  solution $w=w_{c}$ such that  $w_{c}(x)$ is strictly decreasing in $x\in (-\infty, 0]$.  
Furthermore, the mapping $c\mapsto w_{c}(x)$ is strictly decreasing in $[0,\infty)$ for any fixed $x<0$, and is  continuous uniformly in $x\in\R$.
\item[$(b)$] There exists a unique $c_0=c_0(s)\in (0, \infty)$ such that
\[
c_0= \frac{\mu }{2s}  \int_{-\infty}^{0}  w_{c_0}(y)|y|^{-2s} dy.
\]
 \end{itemize}
  \end{theorem}

 The rest of this paper is organized as follows. In Section 2, we consider the semi-wave problem and prove Theorem \ref{teo ell}.   
 In Section 3  we use the semi-wave solution to determine the spreading speed  for the case
 $s\in(\frac12,1)$.  Section 4 is devoted to the case $s\in(0,\frac12]$, where the precise rates of acceleration of the propagation is determined.

\setcounter{equation}{0}
 \section{The semi-wave problem} 
 This section is devoted to the proof of Theorem \ref{teo ell}.
Since the value of the diffusion coefficient $d$ in \eqref{eq 1.1-ell} does not cause any difference in  our proof of Theorem \ref{teo ell},  without loss of generality,  we  assume $d=1$ throughout this section. 

Our strategy is to first consider the following perturbed version of \eqref{eq 1.1-ell} with $\sigma\in (0, 1/4]$:
\begin{equation}\label{eq 1.1-e}
 \left\{
\begin{array}{cll}
 (-\Delta)^s w-c w' = f(w)  \quad
  &{\rm in}\  \, (-\infty,0), 
  \\[2mm]
 \sigma\leq w\leq 1 \quad
  &{\rm in}\  \, \R, 
\\[2mm]
 \displaystyle  w=\sigma \quad  &{\rm in}\  \, [0,+\infty), \\[2mm]
 w(-\infty)=1.
 \end{array}
 \right.
 \end{equation}

We will show that, for every $c\geq 0$, this problem has a solution $\psi_\sigma(x)$ whose limit as $\sigma\to 0$ gives the unique solution of \eqref{eq 1.1-ell},
which we denote by $\psi_0^c$. We will then use \eqref{c0} with $w_c=\psi_0^c$ to uniquely determine $c_0$.

\subsection{A truncated version of the perturbation problem \eqref{eq 1.1-e}} 

For fixed $s\in (1/2, 1)$, $c\geq 0$, $\sigma\in [0, 1/4)$  and $R>1$, we  consider the following problem
\begin{equation}\label{eq 1.1-eR}
 \left\{
\begin{array}{cll}
 (-\Delta)^s w-c w' = f(w)  \quad
  &{\rm in}\  \, I_R:=(-R,0), \\[2mm]
0\leq w\leq 1 \quad
  &{\rm in}\  \, \R, 
\\[2mm]
\displaystyle  w=\sigma \quad  &{\rm in}\  \, [0,+\infty),\\[2mm]
 \displaystyle  w=1 \quad  &{\rm in}\  \, (-\infty,-R].
 \end{array}
 \right.
 \end{equation}

   \begin{proposition}\label{prop-1}
Assume that $\bf (f1)$ is satisfied, and
$s\in(\frac12,1),\  c\geq 0,\ \sigma\in (0,\frac14].$     
  Then \eqref{eq 1.1-eR} has a unique solution  $\psi_{\sigma, R}(x)$, which is strictly decreasing in $(-R,0)$ and for some $C>0$ independent of $R>1$,
  \begin{equation}\label{rho^s}
  \psi_{\sigma, R}(x)-\sigma\leq C|x|^s \mbox{ in } I_R.
  \end{equation}   
  Moreover, 
 
 $(i)$ for any  fixed $\sigma\in[0, 1/4)$,  the mapping $R\mapsto \psi_{\sigma, R}$ is decreasing, i.e.
 \begin{equation}\label{eq 1.1-e0}
  \psi_{\sigma, R_1}(x)< \psi_{\sigma, R_2}(x) \ \ \text{for  $x\in I_{R_1}$  \ \, if $R_1>R_2>0$;}
   \end{equation}
   
  $(ii)$  for any fixed $R>1$,  the mapping $\sigma\mapsto \psi_{\sigma, R}$ is decreasing, i.e.
 \begin{equation}\label{eq 1.1-e00}
  \psi_{\sigma_1, R}(x)< \psi_{\sigma_2, R}(x) \ \ \text{for $x\in I_{R}$  \ \, if $0\leq \sigma_1<\sigma_2<1/4$.}
   \end{equation}

  \end{proposition}

The proof of Proposition \ref{prop-1} relies on some priori estimates. 
  For $\beta>0$, let   $k$ be the greatest integer satisfying $k<\beta$ and set $\beta'=\beta-k\in (0, 1]$.  
  Let $U$ be an open set of $\R^n$. Then we use $\cC^\beta(U)=\cC^{k,\beta'}(U)$ to denote the set of functions $u$ satisfying 
 \[\|u\|_{\cC^\beta(U)}=\|u\|_{C^{k,\beta'}(U)}:=\sum_{l=0}^k\sup_{x\in U}|D^l u(x)|  +\sup_{x,y\in U,\ x\neq y}\frac{|D^ku(x)-D^ku(y)|}{|x-y|^{\beta'}}<+\infty.\]  
 Set
 $$\cC^{\beta}_0 (U):=\{u\in \cC^\beta(\R): u=0\quad {\rm in}\ \R\setminus U\} $$
 and particularly 
 $\cC_0 (U)=\cC^0_0 (U)$. 
 The following result follows directly from \cite[Proposition 2.8]{S07}:
 
   \begin{lemma}\label{Teo-r}
Let $s\in(0,1)$, $U$  a bounded Lipchitz domain in $\R^n$ and $w$  a solution of
\[
(-\Delta)^s w=f\quad\text{in }\, U
\]
with
\[
w\in L^\infty(\mathbb{R}^n),\qquad
f\in \cC^\gamma_{\mathrm{loc}}(U),\ \ \ \gamma\in(0,1).
\]
Then
\[
w\in \cC_{\mathrm{loc}}^{\gamma+2s}(U) \ \mbox{
provided $\gamma+2s\notin\mathbb{N}$.}
\]
 Furthermore,
\[
\|w\|_{\cC^{\gamma+2s}(U')}
\le
C\Big(
\|f\|_{\cC^\gamma(U'')}
+\|w\|_{L^\infty(\mathbb R^n)}
\Big)
\]
for  $U'\Subset U''\Subset U$ and some $C>0$.
\end{lemma}

For $I_R=(-R,0)$ and $I_{\infty}=(-\infty,0)$, we define
\[\rho_R(x)= \mathrm{dist}(x,\partial I_R)=\min\{|x+R|, |x|\} \ \mbox{ and } \rho_\infty(x)=|x| \ \  \mbox{ for } x\in I_\infty.\]

\begin{lemma}\label{pr 2.1}
Suppose   $s\in(\frac12,1)$,     
$\gamma\in(0, 2s-1]$,  $c\geq 0$, $F\in \cC^\gamma_{\loc}(I_R)\cap L^\infty(I_R)$, and $w$ is a classical solution of 
 \begin{equation}\label{eq 1.1-f1}
 \left\{
\begin{array}{cll}
 (-\Delta)^s w-c w' =F  \quad
  &{\rm in}\  \, I_R, 
\\[2mm]
 \displaystyle  w=0 \quad  &{\rm in}\  \, \R\setminus I_R.         
 \end{array}
 \right.
 \end{equation}
Then   
 \begin{equation}\label{reg 1.1}
\|w\|_{\cC^s(\R)}
\le
C\big(  \|F\|_{L^\infty(I_R)} +\|w\|_{L^\infty(\R)}
\big), 
\end{equation}
 and for  $\cI \Subset  \mathcal{J} \Subset   I_R$,   there exists $C=C(\cI,\cJ)>0$ independent of $R$ such that 
 \begin{equation}\label{reg 1.2}
\|w\|_{\cC^{2s}(\cI)} \le C\big( \|w\|_{L^\infty (I_R)}+\|F\|_{L^\infty(\cJ)}  \big)
\end{equation}
and for $\gamma\in(0,2s-1]$  there exists $C=C(\gamma,\cI,\cJ)>0$ independent of $R$ such that 
 \begin{equation}\label{reg 1.3}
\|w\|_{\cC^{2s+\gamma}(\cI)} \le C\big( \|w\|_{L^\infty (I_R)}+\|F\|_{L^\infty(I_R)}+\|F\|_{\cC^\gamma( \mathcal{J})}  \big). 
\end{equation}
\end{lemma}
\begin{proof} Note that (\ref{reg 1.1}) and (\ref{reg 1.2})
 follow from \cite[Proposition 6.6, Lemma 6.3]{K21} respectively.

To show (\ref{reg 1.3}), we write (\ref{eq 1.1-f1}) as 
$$(-\Delta)^s w=c w' +F\quad {\rm in}\ \, I_R,\quad w=0\quad {\rm in}\ \R\setminus I_R. $$
By (\ref{reg 1.2}), 
$c w' +F\in \cC^\gamma_{\loc}(I_R)$. 
 Thus using Lemma \ref{Teo-r} and (\ref{reg 1.1}) we obtain
\begin{align*} 
\|w\|_{\cC^{\gamma+2s}(\cI)} 
&\le C\big(\|w\|_{L^\infty(\R)}+ \|c w' +F\|_{C^\gamma(\mathcal{J})} \big)
\\[1mm]&  \leq C\big( \|w\|_{L^\infty(\R)}+c \| w \|_{\cC^{1+\gamma}(\mathcal{J})} +\|F\|_{\cC^\gamma(\mathcal{J})} 
\big)
\\[1mm]&  \leq C\big( \|w\|_{L^\infty(\R)}+c \| w \|_{\cC^{2s}(\mathcal{J})} +\|F\|_{\cC^\gamma(\mathcal{J})} \big)
\\[1mm]&  \leq C\big( \|w\|_{L^\infty (I_R)}+\|F\|_{L^\infty(I_R)}+\|F\|_{\cC^\gamma( \mathcal{J})}  \big),
\end{align*}
which ends the proof. 
\end{proof}

\medskip

\noindent
{\bf Proof of Proposition \ref{prop-1}:} 

\underline{\it Part 1: Existence}.  Let  $I_R=(-R,0)$ with $R>1$.  Take $\phi_0 \in \cC^3(\R)$ such that $\phi_0$ is non-increasing, 
  $\phi_0=1$ in $(-\infty,-1]$ and $\phi_0=\sigma$ in $[0,+\infty)$. Then 
  $(-\Delta)^s \phi_0-c \phi_0' \in \cC^1(\R)$ and it  is  bounded in $\R$.  

Note that  $\psi=\phi_0+w_0$ is a  solution of   (\ref{eq 1.1-eR}) if 
 $w_0$ is a solution of the problem
 \begin{equation}\label{eq 1.1-e1}
 \left\{
\begin{array}{lll}
 (-\Delta)^s w-c_0 w' =(-\Delta)^s \phi_0-c_0 \phi_0' + f_+(\phi_0+w)  \quad
  &{\rm in}\  \, I_R, 
\\[2mm]
  \phantom{ ------\    }
 \displaystyle  w=0 \quad  &{\rm in}\  \, \R\setminus I_R,        
 \end{array}
 \right.
 \end{equation}
where $f_+=\max\{0,f\}$.

For $s\in (\frac12, 1)$, let  $G_{s,c}(x,y)$ denote the Green function of $ (-\Delta)^s  -c \frac{d}{dx}$  in $I_R$ subject to  zero Dirichlet boundary conditions (\cite{GJZ}), 
and 
\[
\cG_{s,c}[\phi](x):=\int_{I_R}G_{s,c}(x,y)\phi(y)dy,
\]
then
 we can transform \eqref{eq 1.1-e1} into a fixed point problem, and then show the existence of a solution to (\ref{eq 1.1-e1}). 
Indeed, denote 
$$\bT_0 w= \cG_{s,c}[f_+(\phi_0+w)]+ \cG_{s,c}[(-\Delta)^s \phi_0-c \phi_0' ]\quad{\rm for}\, \  w\in \bX_0,$$
 where
 $$\bX_0:=\big\{v\in C_0(I_R)  \!:  \|v\|_{L^\infty(\R)}\leq r_0  \big\}  $$  
and   
 $$r_0:=    \big\| \cG_{s,c}[1] \big\|_{L^\infty(I_R)}\sup_{\tau>0}f_+(\tau)  +\big\|\cG_{s,c}[(-\Delta)^s \phi_0-c \phi_0' ]\big\|_{L^\infty(I_R)}. $$
 Then 
$$\bT_0(\bX_0)\subset  \bX_0,  $$
since for $\phi\in \bX_0$, 
\begin{align*}
\|\bT_0\phi\|_{L^\infty(\R)}&=   \Big\| \cG_{s,c}[ f_+(\phi_0+\phi)]+ \cG_{s,c}[(-\Delta)^s \phi_0-c \phi_0' ]     \Big\|_{L^\infty(\R)}
   \leq r_0.
\end{align*}
It is also easily seen that $\bT_0$ is continuous on $\bX_0$.
By Proposition \ref{pr 2.1},
\begin{align*}
\|\bT_0\phi\|_{C^s(\R)}& \leq C \Big(   \big\|(-\Delta)^s \phi_0-c \phi_0' + f_+(\phi_0+w) \big\|_{L^\infty(I_R)} +\|\phi\|_{L^\infty(\R)}
\big)  \Big)
\\[1mm]&\leq   C' \Big( \|\phi_0\|_{C^2(\R)} + \sup_{\tau>0}f_+(\tau)+r_0  \Big), 
\end{align*}
 and so it follows  from   the Arzel\'a–Ascoli theorem  that 
$\bT_0(\bX_0)$ is a compact subset of $ \bX_0$.  
We may now apply the Schauder fixed point theorem to conclude that there exists $w_0\in \bX_0$ such that 
$$w_0=\bT_0 w_0=\cG_{s,c_0}[f_+(\phi_0+w_0)]+ \cG_{s,c_0}[(-\Delta)^s \phi_0-c_0 \phi_0' ]. $$
Furthermore,  let   $\gamma\in(0,\min\{2s-1,2-2s\})$ and  $\tilde \gamma:=\gamma+2s-1\in(\gamma,1)$, then for  $\cI \Subset  \mathcal{J} \Subset   I_R$, we 
have 
\begin{align*}
\|\bT_0w_0\|_{\cC^{1,\tilde\gamma} (\cI)}&  \leq  C\Big( \|w_0\|_{L^\infty (I_R)}+\|f_+(\phi_0+w_0)]+  (-\Delta)^s \phi_0-c_0 \phi_0' \|_{L^\infty(I_R)}
\\[1mm]&\qquad+\|f_+(\phi_0+w_0)]+  (-\Delta)^s \phi_0-c \phi_0' \|_{C^\gamma( \mathcal{J})}  \Big)  
\\[1mm]&\leq  C'\Big( r_0+\sup_{\tau>0}f_+(\tau) + \| \phi_0\|_{\cC^2(\R)}+\sup_{\tau>0}|f_+'(\tau)| \|w_0\|_{\cC^\gamma(I_R)}\Big)
\\[1mm]&\leq  C'\Big( r_0+\sup_{\tau>0}f_+(\tau) + \| \phi_0\|_{\cC^2(\R)}+\sup_{\tau>0}|f_+'(\tau)| \|w_0\|_{\cC^s(I_R)}\Big)
\\[1mm]&\leq  C' \Big( r_0+\sup_{\tau>0}f_+(\tau) + \| \phi_0\|_{\cC^2(\R)}+\sup_{\tau>0}|f_+'(\tau)| \big( \|\phi_0\|_{\cC^2(\R)} + \sup_{\tau>0}f_+(\tau)+r_0  \big) \Big)
\\[1mm]&\leq  C' \Big( (\sup_{\tau>0}|f_+'(\tau)|+1)(r_0+\| \phi_0\|_{\cC^2(\R)})+\sup_{\tau>0}f_+(\tau) 
 +\sup_{\tau>0}|f_+'(\tau)|   \sup_{\tau>0}f_+(\tau) \Big), 
\end{align*}
and hence
 $w_0\in C_{\loc}^{1,\tilde\gamma} (I_R)\cap C^s_0(I_R)$, so  $w_0$ is a classical solution of (\ref{eq 1.1-e1}).  By \cite{GJZ}, $$G_{s}(x,y)\leq C\rho_R(x)^s|x-y|^{s-1}\quad \text{in $I_R\times I_R$,} $$ 
and so from the definition of $\bT_0$ we obtain
$$w_0(x)=(\bT w_0)(x)\leq C\rho_R(x)^s \mbox{ for } x\in I_R, $$
where $\rho_R(x)=d(x, \partial I_R)=\min\{|x+R|, |x|\}$, and $C$ is some positive constant independent of $R>1$.
 As a consequence, we see that 
 $\psi_{\sigma,R}:=w_0+\phi_0$ is a classical solution of 
$$
 \left\{
\begin{array}{cll}
 (-\Delta)^s w-c_0 w' = f_+(w)  \quad
  &{\rm in}\  \, (-R,0),  
\\[2mm]
 \displaystyle  w=1 \quad  &{\rm in}\  \, (-\infty,-R], \\[2mm]
 \displaystyle  w=\sigma \quad  &{\rm in}\  \, [0,+\infty).        
 \end{array}
 \right.
 $$
It follows from the maximum principle that 
 $$\sigma\leq  \psi_{\sigma,R} \leq 1\quad {\rm in}\ \R, $$
 which implies that $\psi_{\sigma,R}$ is a classical solution of  (\ref{eq 1.1-eR}). Moreover, \eqref{rho^s} holds. \medskip

 \underline{\it Part 2: Monotonicity,  uniqueness and \eqref{eq 1.1-e0}. } First we claim that $\sigma<\psi_{\sigma,R}<1$ in $I_R$. 
  If there exists $x_0\in I_R$ such that $\psi_{\sigma,R}(x_0)=\max_{\R} \psi_{\sigma,R}\geq 1$,
 then 
 $$ 0\geq f(\psi_{\sigma,R}(x_0))=  (-\Delta)^s \psi_{\sigma,R}(x_0)-c_0 \psi_{\sigma,R}' (x_0)=(-\Delta)^s \psi_{\sigma,R}(x_0)>0,$$ 
 which is impossible. 
 A similar contradiction arises if there exists $y_0\in I_R$ such that $\psi_{\sigma,R}(y_0)=\min_{\R} \psi_{\sigma,R}\leq \sigma$. 

Secondly,  we show that $\psi_{\sigma,R}(x)$ is non-increasing in $x$.  If not,  there exits a point $z_0\in I_R$ such that for some small $\delta_0>0$, 
$$\psi_{\sigma,R}(x)< \psi_{\sigma,R}(z_0) \quad {\rm for}\ \, x\in (z_0-\delta_0,z_0)$$

For $\tau\geq 0$, we define
$$w_\tau(x)=\psi_{\sigma,R}(x+\tau)\quad\text{ in $\R$}; $$ 
then 
\begin{equation}\label{eq 1.1-e2}
 \left\{
\begin{array}{lll}
 (-\Delta)^s w_\tau-c_0 w_\tau' = f(w_\tau)  \quad
  &{\rm in}\  \, (-R-\tau,-\tau), 
\\[2mm]
 \phantom{ ------   }
 \displaystyle  w=1 \quad  &{\rm in}\  \, (-\infty,-R-\tau], \\[2mm]
 \phantom{ ------   }
 \displaystyle  w=0 \quad  &{\rm in}\  \, [-\tau,+\infty).      
 \end{array}
 \right.
 \end{equation}
When $\tau\geq  R$, it is obvious that 
\begin{equation}\label{eq 1.1-e3}
w_\tau=\sigma< \psi_{\sigma,R}\quad {\rm in}\  I_R. 
 \end{equation}
So the  graph of $w=w_\tau(x)$ is below that of $w=\psi_{\sigma, R}(x)$ in the $x-w$ plane when $\tau\geq R$. As we decrease the value of $\tau$ from $R$ (i.e., we slide the graph of $w_\tau$ rightwards),
it is easily seen from the existence of the above point $z_0$ that there will be a first $\tau$ value less than $R$ but greater than $\delta_0$, say $\tau=\tau_0\in (\delta_0, R)$, such that 
the graph of $w=w_\tau(x)$ lies below that of $w=\psi_{\sigma, R}(x)$ for $\tau>\tau_0$, but the two graphs touch at some point $x_0\in I_R$ when $\tau=\tau_0$,
 i.e.,
\begin{equation}\label{eq 1.1-e4} 
w_{\tau_0}(x_0)=\psi_{\sigma,R}(x_0), \quad  w_{\tau_0}(x) \leq \psi_{\sigma,R}(x)\ \, {\rm for}\ x\in\R.  
 \end{equation}
Let $v_{\tau_0}=\psi_{\sigma,R}-w_{\tau_0}$; then $x_0$ is a global minimum point of $v_{\tau_0}\not\equiv 0$, which leads to
\begin{align} \label{eq 1.1-e4-1-1}
 0=f(\psi_{\sigma,R}(x_0))-f(w_{\tau_0}(x_0))=(-\Delta)^s v_{\tau_0}(x_0)-c_0 v_{\tau_0}'=(-\Delta)^s v_{\tau_0}(x_0)<0,
\end{align} 
a contradiction. 
This proves the non-increasing monotonicity of $\psi_{\sigma, R}(x)$ in $x$. 
\smallskip

Next we show the uniqueness of the solution to \eqref{eq 1.1-eR}.  Suppose  that $u_1$ is an arbitrary  solution of (\ref{eq 1.1-eR}). 
We first show that 
$$u_1\leq \psi_{\sigma,R}\quad {\rm in}\ I_R. $$
We again use the sliding method. Set
$$w_\tau(x)=u_1(x+\tau)\quad\text{ in $\R$}; $$ 
and then continuously reduce the value of $\tau$ from $\tau=R$ to $\tau=0$. Clearly $w_\tau\leq \psi_{\sigma, R}$ for $\tau=R$. If this holds for all $\tau\in [0, R]$, then the desired inequality is proved. Otherwise,
there exists $\tau_0\in (0, R)$ such that  
$$w_{\tau_0}\leq  \psi_{\sigma,R}\quad {\rm in}\ I_R \quad \mbox{ and for some }\ x_0\in I_R,\   w_{\tau_0}(x_0)= \psi_{\sigma,R}(x_0),$$
which lead to  a contradiction at $x_0$ as in (\ref{eq 1.1-e4-1-1}). Hence $u_1\leq \psi_{\sigma,R}\quad {\rm in}\ I_R$. Analogously we can slide $\psi_{\sigma, R}(x)$ to show 
$u_1\geq \psi_{\sigma,R}\quad {\rm in}\ I_R$. Thus $u_1\equiv \psi_{\sigma,R}\quad {\rm in}\ I_R$. This proves the desired uniqueness conclusion.

\smallskip

We now show (\ref{eq 1.1-e0}).  For $R_1>R_2>0$, we let 
$$w_\tau(x)=\psi_{\sigma, R_1}(x+\tau)\quad\text{ in $\R$}.  $$
Then for $\tau\geq R_2$, 
$$
w_{\tau}(x)<\psi_{\sigma, R_2}(x)\ \  {\rm in}\ \,  I_{R_2}.    
$$
We can show that the above inequality remains valid for all $\tau\in [0, R_2]$, for otherwise, the same contradiction arises.  

Finally, we prove (\ref{eq 1.1-e00}), again by a sliding argument.  For $\sigma_2>\sigma_1>0$,  let 
$$w_\tau(x)=\psi_{\sigma_1, R}(x+\tau)\quad\text{ in $\R$};  $$
then for $\tau\geq R$, 
$$
w_{\tau}(z)<\psi_{\sigma_2, R}(z)\ \  {\rm in}\ \,  I_{R}.    
$$
We now reduce the value of $\tau$ continuously from $\tau=R$ to  $\tau=0$, and claim the above inequality stays valid. Otherwise, we can deduce the same contradiction as above.  
  \qed

\subsection{The perturbation problem \eqref{eq 1.1-e}} 

In this subsection, we solve the  perturbation problem \eqref{eq 1.1-e}.  We say $w\in C(\R)$ is a maximal solution of  	(\ref{eq 1.1-e})  if  
$$\mbox{$w$ solves \eqref{eq 1.1-e} and } \ w\geq v\ {\rm in}\ \R\  \text{for any solution $v$ of 	(\ref{eq 1.1-e})}. $$		 

  \begin{proposition}\label{prop-2}
Assume that  
$s\in(\frac12,1),\  c\geq 0,\ \sigma\in(0, 1/4]$ and $\bf (f1)$ is satisfied.  Then \eqref{eq 1.1-e} has a maximal solution $w=\psi_{\sigma}$ which  is strictly decreasing in $I_\infty:=(-\infty,0)$.   
  Moreover,  
 \begin{equation}\label{eq 1.1-f0}
  \psi_{\sigma_1}(x)< \psi_{\sigma_2}(x) \ \ \text{for $x\in I_{\infty}$   if $0< \sigma_1<\sigma_2\leq 1/4$.}
   \end{equation}
 \end{proposition}
\begin{proof}  Let 
$\psi_{\sigma,R}$ be the unique solution of (\ref{eq 1.1-eR}).  Since the mapping $R\mapsto \psi_{\sigma, R}$ is non-increasing and   $\psi_{\sigma,R}[\R]= [\sigma,1]$,  the limit 
$$\psi_{\sigma}(x)=\lim_{R\to+\infty}\psi_{\sigma,R}(x)\ \mbox{ exists for every $x$ in}\ \, \R.  $$
Moreover,
\begin{equation}\label{est-0}
\psi_{\sigma}(\R)\subset [\sigma,1],\qquad  \psi_{\sigma}\leq \psi_{\sigma,R}\ \ \text{in $\R$ for each $R>1,$}
\end{equation}
and
$\psi_{\sigma}(x)$ is non-increasing in $\R$. 
For clarity, we divide the proof below into several steps.\smallskip

{\bf Step 1:} We show that $\psi_{\sigma}$ is a solution of  (\ref{eq 1.1-e}).

It follows from Proposition \ref{pr 2.1} that 
 for any $\cI\Subset \cI_1\Subset\cI_2\Subset (-\infty,0)$ there exists $C=C(\cI_1,\cI_2)>0$, independent of $R$,      such that, whenever $\cI_2\Subset I_R$, we have 
\begin{align*}
\|\psi_{\sigma, R} \|_{\cC^{2s} (\cI_1)}   &\leq C\big( \|\psi_{\sigma, R}\|_{L^\infty (\R)}+\|f(\psi_{\sigma, R})\|_{L^\infty(\R)}  \big)  
\\[1mm]&\leq C\big(  1+\sup_{\tau>0}f_+(\tau)\big),
\end{align*}
and 
\begin{align*}
\|\psi_{\sigma,R} \|_{\cC^{2s+\gamma} (\cI)}   &\leq C\big( \|\psi_{\sigma, R}\|_{L^\infty (\R)}+\|f(\psi_{\sigma, R})\|_{L^\infty(\R)} + \|f(\psi_{\sigma, R}) \|_{C^\gamma( \mathcal{I}_1)} \big)  
\\[1mm]&\leq C\big(  1+\sup_{\tau>0}f_+(\tau) +L \|\psi_{\sigma, R} \|_{\cC^{\gamma} (\cI_1)}   \big)
\\[1mm]&\leq C\big(  1+\sup_{\tau>0}f_+(\tau) +L \|\psi_{\sigma, R} \|_{\cC^{2s} (\cI_1)}   \big)
\\[1mm]&\leq C\Big(  1+\sup_{\tau>0}f_+(\tau)+L\big(1+\sup_{\tau>0}f_+(\tau)\big)\Big), 
\end{align*} 
where $L$ is the Lipschitz constant of $f|_{[0, 1]}$.
By (\ref{est-0}) and \eqref{rho^s} we further derive that for some $C>0$, 
$$0\leq \psi_{\sigma} (x)- \sigma \leq C|x|^s\quad {\rm for}\ \,x\in[-1,0].$$
Hence $\psi_{\sigma}\in C(\R)$ and  is a classical solution of  
\begin{equation}\label{eq 1.1-ell-inf-0}
 \left\{
\begin{array}{lll}
 (-\Delta)^s w-c w' = f(w)  \quad
  &{\rm in}\  \, (-\infty,0), 
  \\[2mm]
 \phantom{ ----\   }
 \sigma\leq w\leq 1 \quad
  &{\rm in}\  \, \R, 
\\[2mm]
 \phantom{ ------\,   }
 \displaystyle  w=\sigma \quad  &{\rm in}\  \, [0,+\infty). 
  \end{array}
 \right.
 \end{equation}
 
 We show next that
 $$  \lim_{x\to-\infty}\psi_{\sigma}(x)=1.$$
 By the monotonicity of $\psi_{\sigma}(x)$ in $x$,  the limit $\lim_{|x|\to+\infty}\psi_{\sigma}(x)=\tau_0\in[\sigma,1]$ always exists.
 Now we set 
 $$u_n(x):=\psi_{\sigma}(x-n),$$
 and for any interval $I\subset \R$, let $n_0>0$ be such that $I-2n_0\subset (-\infty,0)$.
 Then as in the argument above we can use the estimates for $u_n$ as a solution of \eqref{eq 1.1-ell-inf-0} with $x=0$ replaced by $x= n$ 
 to see that
 the limit $u_n$, denoted by $u_\infty$,  is   a classical solution of 
 \begin{equation}\label{eq 1.1-ell-inf-1}
 \left\{
\begin{array}{lll}
 (-\Delta)^s w-c w' = f(w)  \quad
  &{\rm in}\  \, \R, 
  \\[2mm]
 \phantom{ ----\   }
 \sigma\leq w\leq 1 \quad
  &{\rm in}\  \, \R.
  \end{array}
 \right.
 \end{equation}
 On the other hand, we necessarily have 
 $u_\infty\equiv \tau_0\in [\sigma, 1]$, which implies from the above equation for $u_\infty$ that $f(\tau_0)=0$.
 We thus must have $\tau_0=1$.
 As a consequence, $\psi_{\sigma}$ is a solution of  (\ref{eq 1.1-e}). \smallskip

{\bf Step 2:}  We prove the strictly decreasing monotonicity of $\psi_\sigma(x)$ in $x\in (-\infty, 0]$.

For $\tau\geq 0$, set 
$v_\tau(x)=\psi_{\sigma}(x+\tau)\mbox{ for } x\in \R$.
By the non-increasing monotonicity, we can see that 
$v_\tau(x)\leq \psi_{\sigma}(x)\quad{\rm for \ } x\in \R$.
Let $w:=\psi_{\sigma}-v_\tau$ in $\R$;
then if for some $\tau>0$,  there exists $x_0\in(-\infty,0)$ such that 
$v_\tau(x_0)= \psi_{\sigma}(x_0)$, i.e., $w$ achieves its minimum at $x=x_0$, then 
$v_\tau'(x_0)=\psi_{\sigma}'(x_0)$ and 
$$ 0>(-\Delta)^s w(x_0)=f(\psi_{\sigma}(x_0))-f(v_\tau(x_0))=0, $$
which is impossible. This proves the strictly decreasing monotonicity of $\psi_\sigma(x)$.\smallskip

{\bf Step 3:}  We show that $\psi_{\sigma}$ is the maximal solution. 

Let $\tilde u$ be a solution of (\ref{eq 1.1-e}). 
 We only need to show for each $R>1$, $\psi_{\sigma,R}\geq \tilde u$ in $\R$.   We use the sliding method again. 
 Set $\bar u_\tau(x)=\psi_{\sigma,R}(x-\tau)$ for $\tau>0$; then clearly $\bar u_\tau> \tilde u$ in $(-\infty,0)$ for $\tau\geq R$.
 Now we decrease $\tau$ continuously towards zero. If there is $\tau_0\in(0,R)$ such that for $\tau >\tau_0$,
 $\bar u_\tau> \tilde u$ in  $(-\infty,0)$, but  $\bar u_{\tau_0}(x)\geq \tilde u(x)$ in  $(-\infty,0)$ and equality holds at some $x=x_0\in (-\infty, 0)$, then  as before  a contradiction appears when evaluating the equation satisfied by $\bar w:=\bar u_{\tau_0}(x)-\tilde u$ 
 at $x=x_0$. Thus $\bar u_\tau> \tilde u$ in  $(-\infty,0)$ for all $\tau\in (0, R]$, and hence  $\tilde u\leq \psi_{\sigma,R}$ in $\R$. Letting   $R\to+\infty$, we  get $\tilde u\leq\psi_{\sigma}$ in $\R$, as desired. 
 \smallskip

  {\bf Step 4:}  We prove (\ref{eq 1.1-f0}). 
  
   Letting $R\to\infty$ in (\ref{eq 1.1-e00}) we obtain
\begin{equation}\label{eq 1.1-f00}
  \psi_{\sigma_1}(x)\leq  \psi_{\sigma_2}(x) \ \ \text{in $I_{\infty}$   if $0\leq \sigma_1<\sigma_2\leq 1/4$.}
   \end{equation}
If there exists a point $x_0\in\R$ such that $\psi_{\sigma_1}(x_0)=  \psi_{\sigma_2}(x_0)$, then necessarily $x_0<0$ and
 we get a contradiction by evaluating the equation satisfied by $ w:=\psi_{\sigma_1}- \psi_{\sigma_2}$ 
 at $x=x_0$.     
\end{proof}

\subsection{ Proof of Theorem \ref{teo ell} } 

For clarity we divide the proof into several steps.

{\bf Step 1:} We prove the existence of a solution to \eqref{eq 1.1-ell}.

We will use the maximal solution
$\psi_{\sigma}$  of (\ref{eq 1.1-e})  obtained in Proposition \ref{prop-2}.  Since $\psi_{\sigma}(\R)= [\sigma,1]$ with $\sigma\in(0,\frac14]$.  By the monotonicity of $\psi_\sigma(x)$, for each $\sigma\in (0, 1/4]$, 
there exists a unique
  $x_\sigma\in(-\infty,0)$ such that 
 $$\psi_{\sigma}(x_\sigma)=\frac12, \ \mbox{ and }  x_{\sigma_1}< x_{\sigma_2}<0\quad {\rm if}\ 0<\sigma_1<\sigma_2\leq 1/4.$$
Hence the following limit always exists:
   $$\lim_{\sigma\to0^+}x_{\sigma}= \bar x\in[-\infty,x_{\frac14}).  $$

 Now we set 
 $$\phi_{\sigma}(x):= \psi_{\sigma}(x+x_\sigma);$$
 then $w=\phi_\sigma$ satisfies
$$
 \left\{
\begin{array}{lll}
 (-\Delta)^s w-c w' = f(w)  \quad
  &{\rm in}\  \, (-\infty,-x_\sigma), 
  \\[2mm]
 \phantom{ ----\   }
 \sigma\leq w\leq 1 \quad
  &{\rm in}\  \, \R, 
\\[2mm]
 \phantom{ ------\,   }
 \displaystyle  w=\sigma \quad  &{\rm in}\  \, [-x_\sigma,+\infty), \\[2mm]
 \phantom{ -\   }
 \displaystyle  w(0)=\frac12,\ \ w(-\infty)=1.
 \end{array}
 \right.
$$

  If $\bar x>-\infty$,   then
 for any $\cI\Subset \cI_1\Subset\cI_2\Subset (-\infty,-\bar x)$, there exists $\tilde \sigma$ such that 
 $$\cI_2\Subset (-\infty,- x_{\tilde \sigma}).  $$
  By Proposition \ref{pr 2.1},
  there exists $C=C(\cI_1,\cI_2)>0$,    independent of  
  $\sigma\in(0,\tilde \sigma]$, such that
\begin{align*}
\|\phi_{\sigma} \|_{\cC^{2s} (\cI_1)}   &\leq C\big( \|\phi_{\sigma}\|_{L^\infty (\R)}+\|f(\phi_{\sigma})\|_{L^\infty(\R)}  \big)  
 \leq C\big(  1+\sup_{\tau>0}f_+(\tau)\big),
\end{align*}
and 
\begin{align*}
\|\phi_{\sigma} \|_{\cC^{2s+\gamma} (\cI)}   &\leq C\big( \|\phi_{\sigma}\|_{L^\infty (\R)}+\|f(\phi_{\sigma})\|_{L^\infty(\R)} + \|f(\phi_{\sigma}) \|_{C^\gamma( \mathcal{I}_1)} \big)  
 \\[1mm]&\leq C\Big(  1+\sup_{\tau>0}f_+(\tau)+L\big(1+\sup_{\tau>0}f_+(\tau)\big)\Big),
\end{align*}
where $L$ is the Lipschitz constant of $f|_{[0,1]}$.
Translating these estimates to $\psi_\sigma$, and using the monotonicity of $\sigma\to \psi_{\sigma}$ and $\psi_{\sigma}[\R]\subset [0,1]$, 
we can use  the Arzel\'a–Ascoli theorem  to conclude that 
$$\psi_{\sigma}\to \psi_{0}\quad {\rm  in}\ \,  C^{2s}_{\loc}((-\infty, 0))\cap L^\infty(\R) \quad {\rm as}\ \, \sigma\to0^+. $$
Moreover, as in Step 2 of the proof of Proposition \ref{prop-2},  we can similarly show that $x\to \psi_{0}$ is strictly decreasing in $(-\infty, 0)$  and 
$\displaystyle \lim_{x\to-\infty}\psi_{0}(x)=1$ due to $\psi_{0}(x)\geq \frac12$ in $(-\infty, -\bar x)$. As a consequence, $\psi_{0}$ is a solution of (\ref{eq 1.1-ell}).  We can also show that $\psi_{0}$ is a maximal solution of \eqref{eq 1.1-ell} as in Step 3 of the proof of  Proposition \ref{prop-2}.

\medskip
Next we exclude the possibility that $\bar x=-\infty$.  If this case happens,  then we can similarly show that 
$$\psi_{\sigma}\to \psi_0\quad {\rm  in}\ \,  C^{2s}_{\loc}(\R)\cap L^\infty(\R) \quad {\rm as}\ \, \sigma\to0^+, $$
and $w=\psi_0$ satisfies
\begin{equation}\label{eq 1.1-wR}
 \left\{
\begin{array}{cll}
 (-\Delta)^s w-c w' = f(w)  \quad
  &{\rm in}\  \, \R,
  \\[2mm]
 0\leq w\leq 1 \quad
  &{\rm in}\  \, \R, 
\\[2mm]   
w(0)=\frac12. 
 \end{array}
 \right.
 \end{equation}
 Moreover, using the monotonicity of $\psi_0(x)$ and the maximum principle, we can further prove $0< \psi_0(x)< 1$ in $\R$, 
and 
$$\lim_{x\to-\infty}\psi_{0}(x)=1,\qquad  \lim_{x\to+\infty}\psi_{0}(x)=0. $$

Now we consider the solution $u(t,x)$ of 
\begin{equation}\label{hq 1.1-wR}
 \left\{
\begin{array}{lll}
 \partial_tu+ (-\Delta)^s w  = f(w)  \quad
  &{\rm in}\  \, (0,+\infty)\times \R, 
   \\[2mm]   \phantom{ ----   }
u(0,\cdot)=u_0\quad &{\rm in}\ \R, 
 \end{array}
 \right.
 \end{equation}
 where  $u_{0}=\psi_{0}\eta_0$ in $\R$,  and $\eta_0\in C^2(\R)$ satisfies   
$$\eta_0(\R)\subset [0,1],\quad  \eta_0=1\ \, {\rm in}\  [-1,1]\quad {\rm and} \quad  \eta_0=0\ \, {\rm in} \ \R\setminus [-2,2]. $$ 
 Fix $\tau\in(0,1)$ and let
 $$\Omega_{\tau} (t)=\{x\in\R: u(t,x)\geq \tau \}. $$
 Then it follows from \cite[Theorem 1.2]{CR13} (also see \cite{FY13})
 that  for any $\hat c\in(0, \frac{f'(0)}{1+2s})$,
 $$u(t,x)\to 1\quad \text{uniformly in $\{|x|\leq e^{\hat ct}\}$ as $t\to+\infty$}.   $$
Therefore,  for $t$ large
$$  [-e^{\hat ct}, e^{\hat ct}]\subset  \Omega_{\tau} (t). $$

Let
$$u_1(t,x)=\psi_{0}(x-c_0t)\quad {\rm in}\  \, (0,+\infty)\times \R; $$
then it is a solution of (\ref{hq 1.1-wR})
with $u_1(0,x)=\psi_{0}(x)$ for $x\in\R$. 
By comparison principle, we have that 
$$u_1\geq u\quad {\rm in}\  \, (0,+\infty)\times \R. $$
It follows that
\begin{align*}
0<\tau \leq u(t,e^{\hat ct}) \leq u_1(t,e^{\hat ct})=\psi_{0}(e^{\hat ct}-ct)\to\psi_0(\infty)=0\quad {\rm as}\ t\to+\infty,
\end{align*}
which is impossible. Hence we can rule out  the case $\bar x=-\infty$. \smallskip

\medskip

{\bf Step 2:}   We show the uniqueness of the semi-wave solution.

 	Assume that, in addition to $\psi_{0}$, there exists another solution
	$\Phi$ of \eqref{eq 1.1-ell}. For each fixed $k\in(0,1)$ and $\tau\geq0$, define
	$$\wtd\Phi_{k,\tau}(x):=k\Phi(x+\tau)\quad {\rm for}\ x\in\R.$$
    Then $\wtd\Phi_{k,\tau}$ satisfies
\begin{equation}
	\left\{
	\begin{array}{cll}
		(-\Delta)^s \wtd\Phi_{k,\tau}-c_0 \wtd\Phi_{k,\tau}'
		=
		kf(\Phi(\cdot+\tau))
		\leq f(\wtd\Phi_{k,\tau})
		&{\rm in}\ (-\infty,-\tau),
		\\[2mm] 
		0\leq \wtd\Phi_{k,\tau}\leq k
		&{\rm in}\ \R,
		\\[2mm]		
		\wtd\Phi_{k,\tau}=0
		&{\rm in}\ [-\tau,+\infty),
		\\[2mm]		
		\displaystyle\wtd\Phi_{k,\tau}(-\infty)=k.
	\end{array}
	\right.
\end{equation}

	For  $\tau\geq0$ and fixed  $k\in(0,1)$, define
	\[
	P_\tau(x)=P_{\tau,k}(x):=
	\psi_{0}(x)-\wtd\Phi_k(x+\tau).
	\]
	Since $\wtd\Phi_k$ is non-increasing, $P_{\tau_2}\geq P_{\tau_1}$ whenever
	$\tau_2\geq\tau_1$. Moreover,
	$P_\tau(-\infty)=1-k>0$ and $P_\tau>0$ in
	$(-\infty,0)$ for all sufficiently large $\tau$.
	Define
	\begin{align}
		\tau_*:=
		\inf\big\{\tau>0:
		P_\tau(x)>0\ \text{for all }x\in(-\infty,0)\big\}.
	\end{align}
	We claim that $\tau_*=0$.
	
	Suppose, to the contrary, that $\tau_*>0$. By continuity,
	$P_{\tau_*}\geq0$ in $\R$, and
	$P_{\tau_*}\not\equiv0$ since
	$P_{\tau_*}(-\infty)=1-k>0$. As additionally
	$P_{\tau_*}(x)=\psi_{0}(x)>0$ for $x\in(-\tau_*,0)$, the definition of $\tau_*$ implies the existence of 	$x_0\in(-\infty,-\tau_*)$ such that $P_{\tau_*}(x_0)=0$. 
	 Hence,
	\[
	P_{\tau_*}'(x_0)=0.
	\]
	For $x<-\tau_*$, one has
	\[
	(-\Delta)^sP_{\tau_*}-cP_{\tau_*}'
	\geq
	f(\psi_{0})-f\big(k\Phi(\,\cdot+\tau_*)\big),
	\]
	which leads to, using
	$\psi_0^{c_0}(x_0)=k\Phi(x_0+\tau_*)$ and $P_{\tau_*}'(x_0)=0$,
	\begin{align*}
	(-\Delta)^sP_{\tau_*}(x_0)\geq0.
	\end{align*}
	 On the other hand, since
	$P_{\tau_*}\geq0$ in $\R$, $P_{\tau_*}(x_0)=0$, and
	$P_{\tau_*}\not\equiv0$,
	\[
	(-\Delta)^sP_{\tau_*}(x_0)
	=
	-c_{s}\int_{\R}
	\frac{P_{\tau_*}(y)}{|x_0-y|^{1+2s}}\,dy
	<0,
	\]
	which is a contradiction. Therefore, $\tau_*=0$.
	
	It follows that $P_\tau>0$ in $(-\infty,0)$ for every $\tau>0$.
	Letting $\tau\to0^+$ yields
	$\psi_{0}(x)-k\Phi(x)\geq0$ for all $x\in\R$. Since this inequality
	holds for every $k\in(0,1)$, letting $k\to1^-$ gives
	\[
	\psi_{0}(x)\geq\Phi(x)
	\qquad\text{for all }x\in\R.
	\]
	Reversing the roles of $\psi_{0}$ and $\Phi$ yields the opposite
	inequality. Consequently, $\Phi\equiv\psi_{0}$ in $\R$.

Moreover, it follows from \eqref{rho^s} that 
$$ \psi_0(x)\leq \min\big\{1,\, C|x|^s\big\}\quad {\rm in}\ (-\infty, 0).   $$
\medskip

To stress the dependence of $\psi_0$ on $c$, we will write $\psi_0=\psi_0^c$.

{\bf Step 3:}   We show the decreasing monotonicity of the mapping $c\mapsto \psi_{0}^c$.

 Let $c_{1}> c_{2}\geq 0$ and $\psi_{i}=\psi_0^{c_i}$ for $i=1,2$. 
 For each fixed $k\in(0,1)$ and $\tau\geq 0$, we denote
	$\wtd\Phi_{k,\tau}(x):=k\psi_1(x+\tau)$ in $\R$. Then 
    \begin{align*}
        (-\Delta)^s \wtd\Phi_{k,\tau}-c_{2} \wtd\Phi_{k,\tau}'
		\leq (-\Delta)^s \wtd\Phi_{k,\tau}-c_{1} \wtd\Phi_{k,\tau}'
		=kf(\psi_{c_{1}}(\cdot+\tau))
		\leq f(\wtd\Phi_{k,\tau}), 
    \end{align*}
 Thus  
    $\wtd\Phi_{k,\tau}$ satisfies
\begin{equation}\label{eq mov-1}
	\left\{
	\begin{array}{cll}
		(-\Delta)^s \wtd\Phi_{k,\tau}-c_{0,2} \wtd\Phi_{k,\tau}'
		\leq f(\wtd\Phi_{k,\tau})
		&{\rm in}\ (-\infty,-\tau),
		\\[2mm] 
		0\leq \wtd\Phi_{k,\tau}\leq k
		&{\rm in}\ \R,
		\\[2mm]		
		\wtd\Phi_{k,\tau}=0
		&{\rm in}\ [-\tau,+\infty),
		\\[2mm]		  
		\displaystyle\wtd\Phi_{k,\tau}(-\infty)=k.
	\end{array}
	\right.
\end{equation}
For $\tau\geq0$ and fixed $k\in(0,1)$, denote
	\[
	P_\tau(x)=P_{\tau,k}(x):=
	\psi_{2}(x)-\wtd\Phi_{k,\tau}(x).
	\]
Then we can repeat the argument in Step 2 above to  derive  
$$\psi_1(x)\leq \psi_2(x)\quad {\rm for}\ x\in\R.$$
If there exists a point $x_0\in (-\infty, 0)$ such that
$$\psi_1(x_0)= \psi_2(x_0),$$
then the strong maximum principle implies that 
$$\psi_1(x)= \psi_2(x)\quad {\rm for}\ x\in(-\infty,0),$$
which is impossible, since $c_{1}>c_{2}$. 
Therefore, 
$$\psi_1< \psi_2\quad {\rm in}\  (-\infty,0).  $$

\medskip

 {\bf Step 4:}   Finally, we show the continuity of $\psi_{0}^c$ with respect to
$ c\in[0,\infty)$.

Let $c_{n}\to \bar c_{0}\in[0,\infty)$. Using the regularity estimates for $\psi_0^{c_n}$, much as before, after passing to a subsequence if necessary, 
we obtain
$\psi_0^{c_{n}}\to\psi_*$ locally uniformly in $(-\infty,0]$, where
$\psi_*$ satisfies \eqref{eq 1.1-ell} with $c=\bar c_0$.
The uniqueness of the solution to \eqref{eq 1.1-ell} then implies that
$\psi_*=\psi_0^{\bar c_0}$. Since every convergent subsequence has the same
limit, the entire sequence converges to $\psi_0^{\bar c_0}$. Therefore,
$\psi_0^c(x)$ depends continuously on $c\in[0,\infty)$, uniformly for $x$ in any bounded subset of $(-\infty, 0]$.

We show next that this continuity is uniform for $x\in (-\infty, 0]$. Fix $\hat c_0>\bar c_0$. By the monotonicity proved in Step 3 above, for all large $n$ and $ c=c_n \mbox{ or } \bar c_0$,
\[
\psi_0^0(x)\ge \psi_0^{c}(x)\ge \psi_0^{\hat c_0}(x) \mbox{ in }  (-\infty, 0].
\]
Given any $\eps>0$, using $\psi_0^{\hat c_0}(-\infty)=\psi_0^0(-\infty)=1$, we can find $M\gg 1$ so that
\[
\psi_0^{\hat c_0}(x),\ \psi_0^0(x)\in (1-\eps, 1] \mbox{ for } x\in (-\infty, -M].
\]
It follows that
\[
|\psi_0^{c_n}(x)-\psi_0^{\bar c_0}(x)|<\eps \mbox{ for } x\leq -M.
\]
Since $\psi_0^c(x)$ depends continuously on $c\in[0,\infty)$, uniformly for $x$ in $[-M, 0]$, we can find $N\gg 1$ such that
\[
|\psi_0^{c_n}(x)-\psi_0^{\bar c_0}(x)|<\eps \mbox{ for } x\in [ -M, 0],\ n\geq N.
\]
Thus $\|\psi_0^{c_n}-\psi_0^{\bar c_0}\|_{L^\infty(\R_-)}\leq \eps$ for all large $n$, as desired.
\medskip

{\bf Step 5:} We prove the existence of a unique $c_0=c_0(s)>0$ such that \eqref{c0} holds.

Define 
$$\zeta_0(c)= \frac{\mu }{2s}  \int_{-\infty}^{0}  \psi_0^c(y)(-y)^{-2s} dy-c\quad \text{ for $c\geq 0$.}$$
From the decreasing monotonicity, we  have 
$$\psi_0^{c}(x)\leq \psi_0^0(x)\leq \min\big\{1,\, C|x|^s\big\}\quad {\rm in}\  (-\infty, 0)\quad{\rm for\ every}\ c\geq 0.   $$
Hence
$\displaystyle \int_{-\infty}^{0}  \psi_0^c(y)(-y)^{-2s} dy$ is well defined due to $s\in(\frac12,1)$.  
Moreover, by the conclusions proved in Steps 3 and 4 above,  the mapping $c\in[0,+\infty)\mapsto \psi_0^c$ is strictly decreasing and uniformly continuous in $\R$. 
Hence $\zeta_0(c)$ is continuous,  strictly decreasing,
$$\zeta_0(0)= \frac{\mu }{2s} \int_{-\infty}^{0}  \psi_{0}^0(y)(-y)^{-2s} dy>0,$$
 and 
 $$\lim_{c\to+\infty} \zeta_0(c )\leq \lim_{c\to+\infty}\Big( \frac{\mu }{2s}  \int_{-\infty}^{0} \min\big\{1,\, C|y|^s\big\} (-y)^{-2s} dy-c\Big)=-\infty$$
 So there exists a unique $c_0=c_0(s)>0$ such that $\zeta_0(c_0)=0$, i.e.
\begin{equation*}\label{eqv-1}
 \frac{\mu }{2s}  \int_{-\infty}^{0}  \psi_0^{c_0}(y)(-y)^{-2s} dy=c_0.
\end{equation*}
The proof is now complete. 
\qed

\setcounter{equation}{0}
\section{Spreading speed  for $s\in(\frac12,1)$ }

 Let $(u,g,h)$ be the  solution of \eqref{eq 1.1} and suppose that spreading happens, that is
\[
\lim_{t\to\infty}h(t)=-\lim_{t\to\infty} g(t)=\infty,\; \mbox{ and } \lim_{t\to\infty} u(t,x)=1 \mbox{ locally uniformly in $x\in\mathbb R$}.
\]
Let $c_0>0$ be the speed of the semi-wave $w_{c_0}$ in Theorem  \ref{teo ell}.
We are going to show that
\[
\lim_{t\to\infty}\frac{h(t)}{t}=-\lim_{t\to\infty}\frac{g(t)}{t}=c_0,
\]
which is part (iii) of Theorem \ref{teo 1}.

It suffices to show the conclusion for $h(t)$, as $\tilde u(t,x):=u(t,-x)$ satisfies \eqref{eq 1.1} with free boundaries
$x=\tilde h(t):=-g(t),\; x=\tilde g(t):=-h(t)$ and initial function $\tilde u_0(x):=u_0(-x)$.

\begin{lemma}\label{<c0} Under the above assumptions, we have
\[
\limsup_{t\to\infty}\frac{h(t)}{t}\leq c_0.
\]
\end{lemma}
\begin{proof}
Let
  $\psi_0^{c_0}(x)$ be the semi-wave with speed $c_0$ proved in the previous section.
For any given $\epsilon>0$  we define
\[
\delta:=2\epsilon c_0,\; \overline h(t):=(c_0+\delta)t+L,\; \overline u(t,x)=: (1+\epsilon)\psi_0^{c_0}(x-\overline h(t)),
\]
with $L>0$ to be determined. A simple comparison argument with the ODE problem
\[
v'=f(v),\; v(0)=\|u_0\|_\infty
\]
shows that $u(t,x)\leq v(t)$,  and hence  
\[
\limsup_{t\to\infty}u(t,x)\leq 1 \mbox{ uniformly for $x\in \R$}.
\]
Thus there exists $T>0$ large so that
\[
u(T+t,x)\leq 1+\frac\epsilon 2 \ \mbox{ for }\ t\geq 0, \ \ x\in [g(T+t), h(T+t)].
\]
Since $\psi_0^{c_0}(-\infty)=1$,
we can set $L>\max\{h(T),-g(T)\}$ such that   $\overline h(0)=L>h(T)$ and
\begin{equation}\label{u0}
\overline u(0,x)=(1+\epsilon)\psi_0^{c_0}(x-L)>1+\frac \epsilon 2\geq u(T,x)\ \mbox{ for }\, x\in [g(T), h(T)],
\end{equation}

We show next that
\begin{equation}\label{diff-ineq}
\overline u_t+d(-\Delta)^s\overline u\geq  f(\overline u)\quad \text{for $t>0$ and $x\in [g(t+T), \overline h(t)]$},
\end{equation}
 and
\begin{equation}\label{h'}
\overline h'(t)> \mu \int_{g(t+T)}^{\overline h(t)}\int_{\overline h(t)}^\infty \overline u(t,x) |x-y|^{-1-2s}dydx\ \mbox{ for } t>0.
\end{equation}
Indeed,  for $t>0$ and $g(T+t)<x<\overline h(t)$, using  {\bf (f1)} we have
\begin{eqnarray*}
\overline u_t(t,x) \!\!\!\!&&= -(1+\epsilon)(c_0+\delta)(\psi_0^{c_0})'(x-\overline h(t))>-(1+\epsilon) c_0(\psi_0^{c_0})'(x-\overline h(t))\\[1mm]
&& =  (1+\epsilon)\left[-d(-\Delta)^s\psi_0^{c_0}(x-\overline h(t))+f(\psi_0^{c_0}(x-\overline h(t))\right]\\[1mm]
&&=-d(-\Delta)^s \overline u(t,x)+(1+\epsilon)f(\psi_0^{c_0}(x-\overline h(t))\\[1mm]
&&\geq-d(-\Delta)^s\overline u(t,x)+f(\overline u(t,x)).
 \end{eqnarray*}
 This proves \eqref{diff-ineq}.

 To show \eqref{h'} we calculate
\begin{align*}
 &\mu \int_{g(t+T)}^{\overline h(t)}\int_{\overline h(t)}^\infty \overline u(t,x) |x-y|^{-1-2s}dydx\\
 &=(1+\epsilon) \frac{\mu}{2s} \int_{g(t+T)}^{\overline h(t)} \psi_0^{c_0}(x-\overline h(t)) |\overline h(t)- x|^{-2s}dx
\\[1mm] &\leq  \frac{\mu}{2s} (1+\epsilon) \int_{-\infty}^{\overline h(t)}\psi_0^{c_0}(x-\overline h(t)) |\overline h(t)- x|^{-2s}  dx\\
  &= \frac{\mu}{2s}(1+\epsilon) \int_{-\infty}^{0}\psi_0^{c_0}(y) |y|^{-2s}  dx\\
 &=  (1+\epsilon)c_0
 <c_0+\delta
 =\overline h'(t).
 \end{align*}
  Thus \eqref{h'} holds.
Hence we can apply the comparison principle to obtain 
$$h(t+T)\leq \overline h(t)\quad {\rm for}\ t>0,$$
which implies 
$$\limsup_{t\to+\infty}\frac{h(t)}t\leq c_0+\delta=c_0(1+2\epsilon).  $$
Letting $\epsilon\to 0$, we immediately obtain   $\displaystyle\limsup_{t\to\infty} \frac{h(t)}t\leq c_0$.
 \end{proof}

\begin{lemma}\label{>c0} Under the assumptions of Lemma \ref{<c0}, we have
\[
\liminf_{t\to\infty}\frac{h(t)}{t}\geq c_0.
\]
\end{lemma}
\begin{proof} Since $f'(1)<0$, there exists $\delta_0>0$ small such that $f'(u)<0$ for $u\in [1-\delta_0, 1]$.
For any given $\epsilon\in (0, \delta_0]$, we define
\[\begin{cases}
\delta:=2\epsilon c_0,\;\ \   \underline g(t):=-\underline h(t),\ \underline h(t):=(c_0-\delta)t+L,\\
  \underline u(t,x):=(1-\epsilon)\big[\psi_0^{c_0}(x- \underline h(t))+\psi_0^{c_0}(-x-  \underline h(t))-1\big]_+,
  \end{cases}
\]
with $L>0$ a large constant to be determined and $a_+=\max\{a, 0\}$. We also assume that $\psi_0^{c_0}(x)=0$ for $x>0$. It follows that
\[
  \underline u(t, x)=0 \mbox{ for } t\geq 0,\ |x|\geq \underline h(t).\]

For the above $\epsilon>0$, there exists $L_\epsilon\gg 1$ such that 
$$\psi_0^{c_0}(x)\geq 1-\epsilon\quad {\rm for}\ \, x\leq -L_\epsilon, $$
and hence, with $L=L_\epsilon$, we have
 \[  \underline u(t,x)\geq (1-\epsilon)\big( \psi_0^{c_0}(|x|-\underline h(t))-\epsilon\big)
\geq (1-\epsilon)(1-2\epsilon)
 \quad  \mbox{ for } |x|\leq  \underline h(t)- L_\epsilon,\ \,    t>0.
\]
 
We show next that
\begin{equation}\label{--ut}
\underline u_t+(-\Delta)^s \underline u\leq  f(\underline u)
\ \ \mbox{ for } t>0,\ |x|<\underline h(t),  
\end{equation}
and
\begin{equation}\label{h'2}
\underline h'(t)\leq \frac{\mu}{2s}\int_{-\underline h(t)}^{\underline h(t)}   \underline u(t,x)|\underline h(t)-x|^{-2s}dydx
\;\mbox{ for } t>0.
\end{equation}
 
To show (\ref{--ut}), we observe that 
\begin{align*}
   &\quad  \underline u_t(t,x)+(-\Delta)^s \underline u(t,x)\leq \underline u_t(t,x)+(1-\epsilon)\Big[  (-\Delta)^s\psi_0^{c_0}(x-\underline h(t))+(-\Delta)^s\psi_0^{c_0}(-x-\underline h(t)) \Big]   \\[1mm] 
   &=(1-\epsilon) 
    \Big[ \delta(\psi_0^{c_0})'(x-\underline h(t))+\delta(\psi_0^{c_0})'(-x-\underline h(t)) +f(\psi_0^{c_0}(x-\underline h(t))+f(\psi_0^{c_0}(-x-\underline h(t))\Big]
          \\[1mm]&= f(\underline u) +\Delta(t,x),
\end{align*}
where 
$$\begin{aligned}
\Delta(t,x):=&\  (1-\epsilon) 
    \Big[ \delta(\psi_0^{c_0})'(x-\underline h(t))+\delta(\psi_0^{c_0})'(-x-\underline h(t)) +f(\psi_0^{c_0}(x-\underline h(t))+f(\psi_0^{c_0}(-x-\underline h(t))\Big]\\
    & \ \ 
          - f(\underline u(t,x)).
          \end{aligned}$$
Therefore to prove \eqref{--ut}, it suffices to show that $\Delta(t,x)\leq 0$. This inequality for $\Delta(t,x)$ can be shown in  the same way 
as in \cite{DLZ} where a similar inequality (3.10) is proved, by using the properties of $f$ and $\psi_0^{c_0}(x)$.

To show \eqref{h'2}, we define
\[
\cA_{\pm}=\{(t,x)\in (0,+\infty)\times \R_{\pm}:\,  x\in (-\underline h(t), \underline h(t)), \ \underline u(t,x)>0\},
\]
and note that 
 for $(t,x)\in \cA_{+}$,
\begin{align*}
  \big[\psi_0^{c_0}(x-\underline h(t))
	+\psi_0^{c_0}(-x-\underline h(t))-1\big]_+ 
	  =
	\psi_0^{c_0}(x-\underline h(t))
	-\min\big\{\psi_0^{c_0}(x-\underline h(t)),
	1-\psi_0^{c_0}(-x-\underline h(t))\big\}.
\end{align*}
It follows that
\begin{align*}
	 \frac{\mu}{2s} \int_{-\underline h(t)}^{\underline h(t)}
	\frac{\underline u(t,x)}{|\underline h(t)-x|^{2s}}\, dx 
	& \geq
 \frac{\mu}{2s} (1-\epsilon)
	\int_0^{\underline h(t)}
	\frac{\psi_0^{c_0}(x-\underline h(t))}
	{|\underline h(t)-x|^{2s}}\,dx\\
	&\quad
	- \frac{\mu}{2s} (1-\epsilon)
	\int_0^{\underline h(t)}\frac{
	\min\big\{\psi_0^{c_0}(x-\underline h(t)),
	1-\psi_0^{c_0}(-x-\underline h(t))\big\} }{
	|\underline h(t)-x|^{-2s}}\,dx.
\end{align*}
 Recall that
$\psi_0^{c_0}(z)\leq C(-z)^s$ for $z\in[-1,0]$ and some $C>0$. Choose
$l_\epsilon\in(0,1)$ sufficiently small such that
\[
\frac{\mu C}{ 2s(1-s)}l_\epsilon^{1-s}
\leq\frac{\epsilon c_0}{8}.
\]
Since $\psi_0^{c_0}(x)\to1$ as $x\to-\infty$ and $s>1/2$, by enlarging $L=L_\epsilon>0$ we may assume that
\[
\frac{\mu}{2s( 2s-1)}L_\epsilon^{1-2s}
\leq\frac{\epsilon c_0}{4}
\]
and
\[
\frac{\mu}{2s( 2s-1)}
\big(1-\psi_0^{c_0}(-L_\epsilon)\big)
l_\epsilon^{1-2s}
\leq\frac{\epsilon c_0}{8}.
\]
Since $\underline h(t)\geq L_\epsilon$, by the
definition of $c_0$,
\begin{align*}
\frac{\mu}{2s}
	\int_0^{\underline h(t)}
	\psi_0^{c_0}(x-\underline h(t))
	|\underline h(t)-x|^{-2s}\,dx
	&=
	\frac{\mu}{2s}
	\int_{-\underline h(t)}^0
	\psi_0^{c_0}(z)|z|^{-2s}\,dz
	\geq c_0-\frac{\epsilon c_0}{4}.
\end{align*}
On the other hand, we have
\begin{align*}
	&\quad
	\int_0^{\underline h(t)}
	\min\big\{\psi_0^{c_0}(x-\underline h(t)),
	1-\psi_0^{c_0}(-x-\underline h(t))\big\}
	|\underline h(t)-x|^{-2s}\,dx\\
	&\leq
	\int_{\underline h(t)-l_\epsilon}^{\underline h(t)}
	\psi_0^{c_0}(x-\underline h(t))
	|\underline h(t)-x|^{-2s}\,dx\\
	&\quad+
	\int_0^{\underline h(t)-l_\epsilon}
	\big(1-\psi_0^{c_0}(-x-\underline h(t))\big)
	|\underline h(t)-x|^{-2s}\,dx\\
	&\leq
	C\int_0^{l_\epsilon}z^{-s}\,dz
	+\big(1-\psi_0^{c_0}(-L_\epsilon)\big)
	\int_{l_\epsilon}^{\underline h(t)}z^{-2s}\,dz\\
	&\leq
	\frac{C}{1-s}l_\epsilon^{1-s}
	+\frac{1-\psi_0^{c_0}(-L_\epsilon)}
	{2s-1}l_\epsilon^{1-2s}.
\end{align*}
Hence,
\[
\frac{\mu}{2s}
\int_0^{\underline h(t)}
\min\big\{\psi_0^{c_0}(x-\underline h(t)),
1-\psi_0^{c_0}(-x-\underline h(t))\big\}
|\underline h(t)-x|^{-2s}\,dx
\leq\frac{\epsilon c_0}{4}.
\]
Consequently,
\begin{align*}
	 \frac{\mu}{2s} \int_{-\underline h(t)}^{\underline h(t)}
	\frac{\underline u(t,x)}{|\underline h(t)-x|^{2s}}\, dx 
  \geq
	(1-\epsilon)\left(c_0-\frac{\epsilon c_0}{2}\right)
	    > (1-2\epsilon)c_0
	 =c_0-\delta
  =\underline h'(t),
\end{align*}
which proves \eqref{h'2}.
 
Clearly $\underline u(0,x)\leq 1-\epsilon$ for $x\in [-L, L]$ and $\underline u(0,x)=0$ for $|x|\geq L$.
Since spreading happens to $(u,g,h)$, there exists $T\gg 1$ such that
\[
[-L, L]\subset (g(T), h(T)),\ \ u(T,x)\geq 1-\epsilon \mbox{ for } x\in [-L, L].
\]
Therefore we can apply the comparison principle (see Remark 4.4 in \cite{CDN26}) to compare   $(\underline u(t,x), \underline g(t), \underline h(t))$ 
with $(u(T+t), g(T+t), h(T+t))$ to conclude that
\[
[-\underline h(t), \underline h(t)]\subset (g(T+t), h(T+t)),\  \underline u(t,x)\leq u(t,x) \mbox{ for } t>0,\ |x|<\underline h(t).
\]
It follows that
\[
\liminf_{t\to\infty}\frac{h(t)}t=\liminf_{t\to\infty}\frac{h(T+t)}t\geq \lim_{t\to\infty}\frac{\underline h(t)}t=c_0-\delta=c_0(1-2\epsilon).
\]
Letting $\epsilon\to 0$ we obtain $\displaystyle\liminf_{t\to\infty}\frac{h(t)}{t}\geq c_0$. This completes the proof. \end{proof}

 Clearly, part (iii) of Theorem \ref{teo 1} follows directly from the above two lemmas.

\section{ Precise rates of acceleration  for $s\in(0,\frac12]$}

In this section, we prove parts (i) and (ii) of Theorem \ref{teo 1}. The proof is based on the construction of subtle upper and lower solutions.

We start by giving a basic profile function which will be used to generate various lower solutions of the density function. 

\begin{lemma}\label{lem:smooth-profile}
Given $\rho>2$ and $K_2>0$, there exists a nondecreasing function
$\psi\in C^2(\R)$ such that
\begin{equation}\label{eq:psi-prop}
\psi(\zeta)=0\ \ \mbox{ for \ } \zeta\le0,\qquad
\psi(\zeta)=K_2\zeta^\rho\ \ \mbox{ for\  } 0\le\zeta\le\tfrac12,\qquad
\psi(\zeta)=K_2\ \ \mbox{ for \ } \zeta\ge1,
\end{equation}
and for some positive constants $c_\psi,C_\psi$,
\begin{equation}\label{eq:psi-der-bound}
c_\psi\zeta^\rho\le\psi(\zeta)\le K_2,\qquad
0\le\psi'(\zeta)\le C_\psi\zeta^{\rho-1}
\quad \mbox{ for \ } \zeta\in[0,1].
\end{equation}
\end{lemma}

\begin{proof}
Take $\bar\eta\in C^\infty(\R)$ with $0\le\bar\eta\le1$,
$\bar\eta(\zeta)=0$ for $\zeta\le1/2$, $\bar\eta(\zeta)=1$ for $\zeta\ge1$
and $\bar\eta'\ge0$, and then define
\[
\psi(\zeta):=K_2\bigl[\zeta^\rho(1-\bar\eta(\zeta))+\bar\eta(\zeta)\bigr]
\ \mbox{ for } \zeta>0,\qquad
\psi(\zeta):=0\ \mbox{ for } \zeta\le0.
\]
Then $\psi(\zeta)=K_2\zeta^\rho$ on $[0,1/2]$ and $\psi(\zeta)=K_2$ on
$[1,\infty)$. Since $\rho>2$, we know $\psi$ is $C^2$ in $\R$. Moreover,
\[
\psi'=K_2[\rho\zeta^{\rho-1}(1-\bar\eta)+\bar\eta'(1-\zeta^\rho)]\ge0,
\]
and so $\psi$ is nondecreasing, and \eqref{eq:psi-der-bound} holds with, for
example, 
\[
\mbox{$c_\psi=K_2/2^\rho$ and
$C_\psi=\max\{\rho K_2,\,2^{\rho-1}\|\psi'\|_{L^\infty([1/2,1])}\}$.}
\]
The proof is complete.
\end{proof}


For $\eps\in(0,1/8)$ and $L>0$ define
\begin{equation}\label{eq:profile}
\varphi_L(x):=\psi\!\left(\frac{L-|x|}{\eps L}\right)\quad\mbox{ for } x\in\R.
\end{equation}
Then $\varphi_L\in C^2(\R)$  and
 \[
\varphi_L(x)= \begin{cases} K_2 &\mbox{  for } |x|\leq (1-\eps)L,\\
0& \mbox{ for } |x|\geq L.
\end{cases}
\]
This function $\varphi_L$ will be frequently used below as a buiding block for constructing lower solutions.

\subsection{The case $s\in (0,1/2)$}\label{sec:lower}

\begin{lemma}\label{lem:profile}
Let $s\in(0,1/2)$, $\rho>2$, $K_2>0$, $\eps\in(0,1/8)$ and $\varphi_L$ be given by \eqref{eq:profile}. Then the following conclusions hold:
\begin{itemize}
\item[(i)] There is $C_1=C_1(s,\rho,K_2,\eps)>0$ such that
\[
\begin{cases}
\abswen{(-\Delta)^s\varphi_L(x)}\le C_1 L^{-2s}&\mbox{ for } |x|\leq L,\\
0\le(-\Delta)^s\varphi_L(x)\le C_1L^{-2s}&\mbox{ for } |x|\le(1-\eps)L.
\end{cases}
\]
\item[(ii)] There are constants $C_0=C_0(s,\rho,K_2,\eps)>0$ and
$\delta_0=\delta_0(s,\rho,K_2,\eps)\in(0,1/2)$ such that
\begin{equation}\label{eq:profile-tail}
(-\Delta)^s\varphi_L(x)\le -C_0 L^{-2s}
\qquad\text{for } x\in[(1-\eps\delta_0)L,L],\ L>0.
\end{equation}
\end{itemize}
\end{lemma}

\begin{proof}
Denote $\Psi(x):=\psi\!\left(\frac{1-|x|}{\eps }\right)=\psi\!\left(\frac1\eps-\frac{|x|}{\eps}\right)=\varphi_1(x)$. Then $\Psi\in C^2(\R)$ and $\varphi_L(x)=\Psi(x/L)$, 
\begin{equation}\label{eq:homog}
(-\Delta)^s\varphi_L(x)=L^{-2s}(-\Delta)^s\Psi\!\left(\frac{x}{L}\right),
\qquad x\in[-L,L].
\end{equation}
Since $\Psi\in C^2(\R)$ and $0<s<1/2$, the function $(-\Delta)^s\Psi$ is
continuous over $\R$, and hence bounded on $[-1,1]$. This proves (i) with
$C_1=\max_{|z|\le1}|(-\Delta)^s\Psi(z)|$, the additional sign statement
$(-\Delta)^s\varphi_L\ge0$ for $|x|\le(1-\eps)L$ follows from
$\varphi_L\equiv K_2$ (the maximum of $\varphi_L$) there.

For (ii), since $\Psi(x)=0$ for $|x|\geq 1$ and $\Psi=K_2$ on $[0,1-\eps]$, we have
\[
(-\Delta)^s\Psi(1)
=-c_s\int_{-1}^{1}\frac{\Psi(z)}{(1-z)^{1+2s}}\,\dd z
\le-c_sK_2\int_{0}^{1-\eps}\frac{\dd z}{(1-z)^{1+2s}}
=-\frac{c_sK_2}{2s}\left(\eps^{-2s}-1\right)<0.
\]
By the continuity of $(-\Delta)^s\Psi$, there is
$\delta_0\in(0,1/2)$ such that
\[\mbox{$(-\Delta)^s\Psi(z)\le\tfrac12(-\Delta)^s\Psi(1)<0$ for
$z\in[1-\eps\delta_0,1]$.}
\]
 Setting $C_0:=-\tfrac12(-\Delta)^s\Psi(1)>0$ and
using \eqref{eq:homog} gives \eqref{eq:profile-tail}.
\end{proof}

\subsubsection{The lower bound in Theorem~\ref{teo 1} {\rm (i)}}

\begin{lemma}\label{lem:lower}
Under the assumptions of Theorem~\ref{teo 1}, if $s\in (0, 1/2)$ then
\begin{equation}\label{eq:lower}
\liminf_{t\to\infty}\frac{h(t)}{t^{1/(2s)}}\ge
\left(\frac{2^{1-2s}\mu}{1-2s}\right)^{1/(2s)},\qquad
\liminf_{t\to\infty}\frac{-g(t)}{t^{1/(2s)}}\ge
\left(\frac{2^{1-2s}\mu}{1-2s}\right)^{1/(2s)}.
\end{equation}
\end{lemma}

\begin{proof}
It suffices to prove the estimate for $h(t)$; the estimate for $-g(t)$
follows by applying the same argument to the  problem with initial
datum $u_0(-x)$.

Let $\eps\in(0,1/8)$, $\rho>2$, and  $\psi$ be the
smooth function in Lemma~\ref{lem:smooth-profile} with $K_2:=1-\eps$. Set
\begin{equation}\label{eq:params}
K_1:= \frac{\mu(1-\eps)\left[(2-\eps)^{1-2s}-\eps^{1-2s}\right]}{1-2s}.
\end{equation}
For large $\theta>0$ to be chosen, define, for $t\ge0$ and $x\in\R$,
\begin{equation}\label{eq:lower-def}
\underline h(t):=(K_1t+\theta)^{1/(2s)},\qquad
\underline g(t):=-\underline h(t),
\end{equation}
and
\begin{equation}\label{eq:lower-u}
\underline u(t,x):=
\psi\!\left(\frac{\underline h(t)-|x|}{\eps\underline h(t)}\right),
\end{equation}
 Then
\[
\underline u(t,x)=\begin{cases}0 &\mbox{ for } |x|\geq \underline h(t),\\
 K_2 &\mbox{ for } |x|\le(1-\eps)\underline h(t).
\end{cases}
\]
Note that $\underline u(t,x)=\varphi_{\underline h(t)}(x)$ in the
notation of \eqref{eq:profile}, so Lemma~\ref{lem:profile} applies with
$L=\underline h(t)$.

We verify that $(\underline u,\underline g,\underline h)$ is a lower
solution of \eqref{eq 1.1} for $\theta$ sufficiently large.

\smallskip

\textbf{Step 1:} We check the free boundary condition.

Since $\underline h^{2s}(t)=K_1t+\theta$, we have
\begin{equation}\label{eq:hder}
\underline h'(t)=\frac{K_1}{2s}  \,\underline h(t)^{1-2s}.
\end{equation}
Using $\underline u(t,\cdot)= K_2$ on $[-(1-\eps)\underline h(t),(1-\eps)\underline h(t)]$,
we obtain
\begin{align*}
&\mu\int_{-\underline h(t)}^{\underline h(t)}\int_{\underline h(t)}^{\infty}
\frac{\underline u(t,x)}{|x-y|^{2s+1}}\,\dd y\,\dd x
 \ge K_2\frac{\mu}{2s} \int_{\eps\underline h(t)}^{(2-\eps)\underline h(t)}
r^{-2s}\,\dd r 
\\ &=\frac{K_2\mu\,\underline h^{1-2s}(t)
\left[(2-\eps)^{1-2s}-\eps^{1-2s}\right]}{2s(1-2s)} =\frac{K_1}{2s}\,\underline h^{1-2s}(t)
=\underline h'(t),
\end{align*}
where we have used \eqref{eq:params} and $K_2=1-\eps$. Thus
\begin{equation}\label{eq:lower-fb}
\underline h'(t)\le
\mu\int_{-\underline h(t)}^{\underline h(t)}\int_{\underline h(t)}^{\infty}
\frac{\underline u(t,x)}{|x-y|^{2s+1}}\,\dd y\,\dd x.
\end{equation}
By symmetry (that is, $\underline u(t,x)=\underline u(t,-x)$,
$\underline g(t)=-\underline h(t)$), the desired inequality for
the left boundary follows from \eqref{eq:lower-fb}.

\smallskip

\textbf{Step 2:} We prove the desired inequality for the differential equation, namely,
for $\theta$ sufficiently large,
\begin{equation}\label{eq:interior-lower}
\p_t\underline u+d(-\Delta)^s\underline u\le f(\underline u)
\qquad\text{for }t>0,\ x\in(-\underline h(t),\underline h(t)).
\end{equation}
By symmetry it suffices to consider $x\in [0, \underline h(t))$. Since $\psi\in C^2(\R)$ and
$\underline h\in C^1$, the partial derivative $\p_t\underline u$ exists everywhere.
Writing $h=\underline h(t)$ and $\zeta=(h-x)/(\eps h)$, a direct computation
using \eqref{eq:hder} gives, for $(1-\eps)h<x<h$,
\begin{equation}\label{eq:ut}
\p_t\underline u(t,x)
=\frac{K_1}{2s}\,\psi'(\zeta)\,\frac{x}{h}\,\frac1\eps\,h^{-2s}\ \ge0,
\end{equation}
while $\p_t\underline u=0$ for $0\le x<(1-\eps)h$ (and, by
\eqref{eq:psi-prop}, $\p_t\underline u\to0$ as $x\to(1-\eps)h$).

By ${\bf (f1)}$, $u\mapsto f(u)/u$ is non-increasing and positive in $(0,1)$. Thus 
$0\le\underline u\le K_2<1$ implies
\begin{equation}\label{eq:c0}
\sigma_0:=\frac{f(K_2)}{K_2}>0
\qquad\text{and }\qquad f(\underline u)\ge \sigma_0\,\underline u.
\end{equation}

Let $C_1,C_0,\delta_0$ be the constants from
Lemma~\ref{lem:profile}, and set
\begin{equation}\label{eq:eta}
\eta:=\min\left\{\delta_0,\
\left(\frac{dC_0\,(2s)\eps}{K_1C_\psi}\right)^{1/(\rho-1)}
\right\}>0.
\end{equation}
For fixed $t\ge0$ and $h=\underline h(t)$, we consider three regions of $x\in [0, h]$.

\smallskip

\underline{Region 1:} $0\le x\le(1-\eps)h$. In this range, $\underline u\equiv
K_2$ and so $\p_t\underline u=0$. By Lemma~\ref{lem:profile}(i),
$(-\Delta)^s\underline u\le C_1h^{-2s}$. Hence
\[
\p_t\underline u+d(-\Delta)^s\underline u
=d(-\Delta)^s\underline u
\le dC_1h^{-2s}
\le f(K_2)=f(\underline u),
\]
provided $\theta\geq dC_1/f(K_2)$, which implies $h=\underline h(t)\ge(dC_1/f(K_2))^{1/(2s)}$ for all $t\geq 0$.

\smallskip

\underline{Region 2:} $(1-\eps\eta)h\le x<h$ (so
$\zeta\le\eta\le\delta_0$). Since $\eta\le\delta_0$,
Lemma~\ref{lem:profile}(ii) gives
$(-\Delta)^s\underline u\le - C_0h^{-2s}$. By \eqref{eq:ut} and
\eqref{eq:psi-der-bound},
\[
\p_t\underline u
\le\frac{K_1}{2s}C_\psi\eta^{\rho-1}\frac1\eps\,h^{-2s}
\le dC_0h^{-2s}
\le -d\left(-\Delta)^s\underline u\right).
\]
Hence, due to $f(\underline u)\ge0$, we have
\[
\p_t\underline u+d(-\Delta)^s\underline u
\le f(\underline u).
\]

\smallskip

\underline{Region 3:} $(1-\eps)h<x<(1-\eps\eta)h$ (so
$\eta<\zeta<1$). For such $x$, we have $\underline u=\psi(\zeta)\ge c_\psi\zeta^\rho\ge
c_\psi\eta^\rho$, and by \eqref{eq:ut} and \eqref{eq:psi-der-bound},
\begin{equation}\label{eq:ut-over-u}
\frac{\p_t\underline u}{\underline u}
\le\frac{K_1C_\psi}{2sc_\psi}\,\zeta^{-1}\,\frac1\eps\,h^{-2s}
\le\frac{K_1C_\psi}{2sc_\psi\eta}\,\frac1\eps\,h^{-2s}.
\end{equation}
By Lemma~\ref{lem:profile}(i), $|(-\Delta)^s\underline u|\le C_1h^{-2s}$.
Since $h^{-2s}\to0$ as $h\to\infty$, for all sufficiently large $\theta$ we have
\begin{equation}\label{eq:region3-small}
\frac{K_1C_\psi}{2sc_\psi\eta}\,\frac1\eps\,h^{-2s}\le\frac{\sigma_0}{2},
\qquad
dC_1h^{-2s}\le\frac{\sigma_0}{2}\,c_\psi\eta^\rho.
\end{equation}
Consequently, due to  $\underline u\ge c_\psi\eta^\rho$,
\begin{align*}
\p_t\underline u+d(-\Delta)^s\underline u
\le\frac{\sigma_0}{2}\underline u+dC_1h^{-2s}\le\frac{\sigma_0}{2}\underline u+\frac{\sigma_0}{2}\underline u
=\sigma_0\underline u
\le f(\underline u).
\end{align*}

Combining the three regions, \eqref{eq:interior-lower} holds for all
sufficiently large $\theta$ (recall $\underline h(t)\ge\theta^{1/(2s)}$).
This proves that $(\underline u,\underline g,\underline h)$ is a lower
solution.

\smallskip

\textbf{Step 3:} We complete the proof by the comparison principle. 

Since spreading occurs, there is $t_0>0$ large such that
$[\underline g(0),\underline h(0)]=[-\theta^{1/(2s)},\theta^{1/(2s)}]
\subset[g(t_0),h(t_0)]$ and
$ u(t_0,x)\geq 1-\eps=K_2\geq \underline u(0,x)$ for
$x\in[\underline g(0),\underline h(0)]$ (recall that $u(t,\cdot)\to1$
locally uniformly as $t\to\infty$). Applying the comparison principle yields
\[
[\underline g(t),\underline h(t)]\subset[g(t_0+t),h(t_0+t)],
\quad
\underline u(t,x)\le u(t_0+t,x)\ \mbox{ for } t>0,\ x\in\R.
\]
In particular $h(t_0+t)\ge\underline h(t)=(K_1t+\theta)^{1/(2s)}$ for $t>0$, and hence
\[
\liminf_{t\to\infty}\frac{h(t)}{t^{1/(2s)}}
\ge K_1^{1/(2s)}
=\left(\frac{  (1-\eps)\left[(2-\eps)^{1-2s}-\eps^{1-2s}\right]}
{1-2s}\mu\right)^{\frac1{2s}}.
\]
Letting $\eps\to0^+$ gives the first inequality in \eqref{eq:lower}; the
second follows by symmetry. This completes the proof.
\end{proof}

\subsubsection{The upper bound  in Theorem~\ref{teo 1} {\rm (i)}}\label{sec:upper}

\begin{lemma}\label{lem:upper}
Under the assumptions of Theorem~\ref{teo 1}, if $s \in (0, 1/2)$ then
\begin{equation}\label{eq:upper}
\limsup_{t\to\infty}\frac{h(t)}{t^{1/(2s)}}\le
\left(\frac{2^{1-2s}\mu}{1-2s}\right)^{1/(2s)},\qquad
\limsup_{t\to\infty}\frac{-g(t)}{t^{1/(2s)}}\le
\left(\frac{2^{1-2s}\mu}{1-2s}\right)^{1/(2s)}.
\end{equation}
\end{lemma}

\begin{proof}
As before it suffices to prove the estimate for $h(t)$. Fix small
$\eps>0$ and define, for $t\ge0$, $x\in\R$,
\begin{equation}\label{eq:upper-def}
\overline h(t):=(K t+\theta)^{1/(2s)},\qquad
\overline g(t):=-\overline h(t),
\qquad
\overline u(t,x):=1+\eps,
\end{equation}
where
\begin{equation}\label{eq:Kbar}
K:=\frac{ \mu(1+\eps)^2\,2^{1-2s}}{1-2s},
\end{equation}
and $\theta>0$ will be chosen later. We verify that
$(\overline u,\overline g,\overline h)$ is an upper solution.

\smallskip

\textbf{Step 1:} We check  the free boundary condition.

Since $\overline h^{2s}(t)=Kt+\theta$,
$\overline h'(t)=\frac K{2s}\,\overline h^{1-2s}(t)$. Using
$\overline u\equiv1+\eps$, we obtain
\begin{align*}
\mu\int_{-\overline h(t)}^{\overline h(t)}\int_{\overline h(t)}^{\infty}
\frac{\overline u(t,x)}{|x-y|^{2s+1}}\,\dd y\,\dd x
&=\mu(1+\eps)\int_{-\overline h(t)}^{\overline h(t)}
\int_{\overline h(t)}^{\infty}\frac{\dd y\,\dd x}{|x-y|^{2s+1}}\\
&=\frac{\mu(1+\eps)}{2s}\int_{0}^{2\overline h(t)}r^{-2s}\,\dd r\\
&=\frac{\mu(1+\eps)\,2^{1-2s}\,\overline h^{1-2s}(t)}{2s(1-2s)}\le\frac{K}{2s}\,\overline h^{1-2s}(t)
=\overline h'(t),
\end{align*}
where the inequality follows from $1<1+\eps$ and \eqref{eq:Kbar}. Thus
\[
\overline h'(t)\ge
\mu\int_{-\overline h(t)}^{\overline h(t)}\int_{\overline h(t)}^{\infty}
\frac{\overline u(t,x)}{|x-y|^{2s+1}}\,\dd y\,\dd x,
\]
and, by symmetry, $\overline g'(t)$ satisfies the desired corresponding inequality for
the left boundary.

\smallskip

\textbf{Step 2:} We prove the desired inequality for the differential equation.

Since $\overline u\equiv1+\eps$, we have
$\p_t\overline u=(-\Delta)^s\overline u=0$. As
$f(1+\eps)<0$ by ${\bf (f1)}$, we have
\[
\p_t\overline u+d(-\Delta)^s\overline u
=0
\ge f(1+\eps)=f(\overline u).
\]
Thus $(\overline u,\overline g,\overline h)$ is an upper solution.

\smallskip

\textbf{Step 3.} We complete the proof by the comparison principle.  

There is $t_0>0$ such that $u(t_0,x)\le1+\eps$ for
all $x\in[g(t_0),h(t_0)]$. Choosing $\theta$ large so that
$[g(t_0),h(t_0)]\subset[-\overline h(0),\overline h(0)]=
[-\theta^{1/(2s)},\theta^{1/(2s)}]$, the comparison principle then gives
\[
\mbox{$[g(t_0+t),h(t_0+t)]\subset[-\overline h(t),\overline h(t)]$ and  $u(t_0+t,x)\leq \overline u(t,x)$  for $t\ge0,\ x\in\R$.}
\]
Hence $h(t_0+t)\le\overline h(t)=(Kt+\theta)^{1/(2s)}$ and
\[
\limsup_{t\to\infty}\frac{h(t)}{t^{1/(2s)}}
\le K^{1/(2s)}
=\left(\frac{\mu(1+\eps)^2\,2^{1-2s}}{1-2s}\right)^{1/(2s)}.
\]
Letting $\eps\to0$ we obtain \eqref{eq:upper}. The estimate for $-g(t)$
follows by symmetry.
\end{proof}

Clearly Theorem~\ref{teo 1} (i) follows directly from Lemmas~\ref{lem:lower} and \ref{lem:upper}. Let us also note that the lower bound for $u(t,x)$ in the proof of Lemmas~\ref{lem:lower} and the conclusion in Lemmas~\ref{lem:upper}  indicate that the boundary of the level set $\{x: u(t,x)>\sigma\}$ $(0<\sigma<1)$ propagates with the same rate as $h(t)$ and $g(t)$ when $t\to\infty$.

\subsection{The borderline case $s=1/2$}\label{sec:half}

 The case $s=1/2$ requires a few modifications of the arguments
in Section~\ref{sec:lower}. Let $\psi$ be given in
Lemma~\ref{lem:smooth-profile}, with properties
\eqref{eq:psi-prop}--\eqref{eq:psi-der-bound}.

For $L>0$ and $\eta\in(0,h/2)$, we define
\begin{equation}\label{eq:profile-half}
\varphi_{L,\eta}(x):=\psi\!\left(\frac{L-|x|}{\eta}\right) \ \mbox{ for } x\in\R.
\end{equation}
Then 
\[\varphi_{L,\eta}(x)= \begin{cases} K_2 &\mbox{ for  $|x|\leq L-\eta$},\\
0 & \mbox{ for $|x|\geq L$.}
\end{cases}
\]

\begin{lemma}\label{lem:profile-half}
Let $s=1/2$ and  $\psi,\varphi_{L,\eta}$ be defined as above. Then the following conclusions hold:
\begin{itemize}
\item[(i)] There is $C_1=C_1(\psi)>0$ such that
\[\begin{cases} |(-\Delta)^{1/2}\varphi_{L,\eta}(x)|\le C_1\eta^{-1}&\mbox{ for 
$|x|\leq L$,}\\
0\le(-\Delta)^{1/2}\varphi_{L,\eta}(x)\le
C_1\eta^{-1} & \mbox{ for $|x|\le L-\eta$.}
\end{cases}\]
\item[(ii)] There are $C_0=C_0(\psi)>0$ and $\delta_0\in(0,1/2)$ such that
\begin{equation}\label{eq:tail-half}
(-\Delta)^{1/2}\varphi_{L,\eta}(x)\le -C_0\eta^{-1}
\qquad\text{for } |x|\in[L-\delta_0\eta,L].
\end{equation}
\end{itemize}
\end{lemma}

\begin{proof}
(i) By homogeneity, for $|x|\in[L-\eta,L]$,
$(-\Delta)^{1/2}\varphi_{L,\eta}(x)=\eta^{-1}(-\Delta)^{1/2}\psi(\zeta)$
with $\zeta=(L-|x|)/\eta\in[0,1]$. Since $\psi\in C^2(\R)$ with $\psi'$
compactly supported, $(-\Delta)^{1/2}\psi$ is continuous and bounded on
$[0,1]$; hence $|(-\Delta)^{1/2}\varphi_{L,\eta}(x)|\le C\eta^{-1}$ there.
For $|x|\le L-\eta$, $\varphi_{L,\eta}(x)=K_2$ is the maximum, so
$(-\Delta)^{1/2}\varphi_{L,\eta}(x)\ge0$, and
\begin{align*}
(-\Delta)^{1/2}\varphi_{L,\eta}(x)
&=c_{\frac 12}\int_{\R}\frac{K_2-\varphi_{L,\eta}(y)}{|x-y|^{2}}\,\dd y\\
&\le c_{\frac 12}K_2\Bigl[\int_{L-\eta}^{L}\frac{\dd y}{|x-y|^2}
+\int_{-L}^{-L+\eta}\frac{\dd y}{|x-y|^2}
+\int_{|y|>L}\frac{\dd y}{|x-y|^2}\Bigr]
\le C\eta^{-1},
\end{align*}
For $y\in [-L, -L+\eta]\cup [L-\eta, L]$, we have $|x-y|\ge\eta$, and so
\[
\int_{L-\eta}^{L}\frac{\dd y}{|x-y|^2}
+\int_{-L}^{-L+\eta}\frac{\dd y}{|x-y|^2}\leq 2\eta^{-1}.
\]
Moreover,
\[
\int_{|y|>L}|x-y|^{-2}\dd y=(L-x)^{-1}+(L+x)^{-1}\le2\eta^{-1}. 
\]
Hence
\[
0\leq (-\Delta)^{1/2}\varphi_{L,\eta}(x)\le C\eta^{-1} \mbox{ for } |x|\leq L-\eta.
\]

(ii) Since $\psi(\zeta)=0$ for $\zeta\leq 0$, $\psi(\zeta)=K_2\zeta^\rho$ on $[0,1/2]$, and
$\psi(\zeta)=K_2$ on $[1,\infty)$,
\[
(-\Delta)^{1/2}\psi(0)
=-c_{1, \frac 12}\int_0^\infty\frac{\psi(x)}{x^2}\,\dd x<0.
\]
 By the
continuity of $(-\Delta)^{1/2}\psi$, there is
$\delta_0\in(0,1/2)$ such that
\[
\mbox{ $(-\Delta)^{1/2}\psi(\zeta)\le
\tfrac12(-\Delta)^{1/2}\psi(0)<0$ for $\zeta\in[0,\delta_0]$. }
\]
Setting
$C_0:=-\tfrac12(-\Delta)^{1/2}\psi(0)>0$ and using the homogeneity
$(-\Delta)^{1/2}\varphi_{L,\eta}(x)=\eta^{-1}(-\Delta)^{1/2}\psi(\zeta)$
gives \eqref{eq:tail-half}.
\end{proof}

\begin{lemma}\label{lem:half-lower}
Under the assumptions of Theorem~\ref{teo 1}, if $s=1/2$
then
\begin{equation}\label{eq:half-lower}
\liminf_{t\to\infty}\frac{h(t)}{t\ln t}\ge\mu,\qquad
\liminf_{t\to\infty}\frac{-g(t)}{t\ln t}\ge\mu.
\end{equation}
\end{lemma}

\begin{proof}
It suffices to prove the inequality for $h(t)$; the estimate for $-g(t)$ follows by symmetry.
Fix small $\eps,\tilde\eps\in(0,1)$,  $\rho>2$, and let $\psi$ be the
function in  Lemma~\ref{lem:smooth-profile} with $K_2:=1-\eps$. Set
\begin{equation}\label{eq:half-K}
K_1:=(1-\eps)(1-\tilde\eps)\mu,
\end{equation}
and for $\theta>0$ (to be chosen later) define, for $t\ge0$ and $x\in\R$,
\begin{equation}\label{eq:half-lower-def}
\underline h(t):=K_1(t+\theta)\ln(t+\theta),\qquad
\underline g(t):=-\underline h(t),\qquad
\eta(t):=(t+\theta)^{\tilde\eps},
\end{equation}
\begin{equation}\label{eq:half-lower-u}
\underline u(t,x):=\psi\!\left(\frac{\underline h(t)-|x|}{\eta(t)}\right).
\end{equation}
Then 
\[
\ \underline u(t,x):=\begin{cases} 0 &\mbox{ for $|x|\geq \underline h(t)$},\\
 K_2 & \mbox{ for $|x|\le\underline h(t)-\eta(t)$}.
 \end{cases}
 \]
We are going to show that $(\underline u, \underline g, \underline h)$ forms a lower solution of \eqref{eq 1.1} with $s=1/2$.

\smallskip

\textbf{Step 1:} We check the free boundary condition.

Since $\underline h'(t)=K_1(\ln(t+\theta)+1)$, it is enough to prove that
for $\theta$ large,
\begin{equation}\label{eq:half-fb}
\mu\int_{-\underline h}^{\underline h}\int_{\underline h}^{\infty}
\frac{\underline u(t,x)}{|x-y|^{2}}\,\dd y\,\dd x
\ge K_1(\ln(t+\theta)+1).
\end{equation}
Writing $h=\underline h(t)$, $\eta=\eta(t)$ and using
$\underline u\equiv K_2$ on $[-h+\eta,h-\eta]$, we have
\begin{align*}
\mu\int_{-h}^{h}\int_{h}^{\infty}
\frac{\underline u(t,x)}{|x-y|^{2}}\,\dd y\,\dd x
&\ge K_2\mu\int_{-h+\eta}^{h-\eta}\int_{h}^{\infty}
\frac{\dd y\,\dd x}{|x-y|^{2}}\\
&=K_2\mu\int_{\eta}^{2h-\eta}\frac{\dd r}{r}
=K_2\mu\left[\ln(2h-\eta)-\ln\eta\right].
\end{align*}

For $\theta$ large, $2h-\eta\ge h$, so 
\[
\ln(2h-\eta)\ge\ln h=\ln K_1+
\ln(t+\theta)+\ln\ln(t+\theta),
\]
 while $\ln\eta=\tilde\eps\ln(t+\theta)$.
Hence, with $K_2=1-\eps$ and $K_1=(1-\eps)(1-\tilde\eps)\mu$,
\begin{align*}
\mu\int_{-h}^{h}\int_{h}^{\infty}
\frac{\underline u(t,x)}{|x-y|^{2}}\,\dd y\,\dd x
&\ge K_2\mu\left[(1-\tilde\eps)\ln(t+\theta)+\ln\ln(t+\theta)+\ln K_1\right]\\
&=K_1\ln(t+\theta)+K_2\mu\left[\ln\ln(t+\theta)+\ln K_1\right]\\
&\ge K_1\ln(t+\theta)+K_1,
\end{align*}
provided $\theta$ is so large that
$K_2\mu[\ln\ln(t+\theta)+\ln K_1]\ge K_1$ for all $t\ge0$. This proves
\eqref{eq:half-fb}. By symmetry, the desired inequality for $\underline g'(t)$ also follows.
\smallskip

\textbf{Step 2:} We prove  that, for $\theta$ large,
\begin{equation}\label{eq:half-interior}
\p_t\underline u+d(-\Delta)^{1/2}\underline u\le f(\underline u)
\quad \mbox{ for } t>0,\ x\in(-\underline h(t),\underline h(t)).
\end{equation}
By symmetry we only need to consider $x\ge0$. Write $h=\underline h(t)$, $\eta=\eta(t)$
and $\zeta=(h-x)/\eta$. Since $\underline u\le K_2<1$ and $u\mapsto
f(u)/u$ is non-increasing, 
\[
\mbox{$f(\underline u)\ge \sigma_0\underline u$ with
$\sigma_0:=f(K_2)/K_2>0$.}
\]

For $0\le x\le h-\eta$ one has $\underline u\equiv K_2$ and
$\p_t\underline u=0$, by Lemma~\ref{lem:profile-half}(i),
$0\le(-\Delta)^{1/2}\underline u\le C_1\eta^{-1}$, and so
\[
\p_t\underline u+d(-\Delta)^{1/2}\underline u
=d(-\Delta)^{1/2}\underline u
\le dC_1\eta^{-1}\le f(K_2)=f(\underline u)
\]
provided $\eta^{-1}=(t+\theta)^{-\tilde\eps}\le f(K_2)/(dC_1)$, which holds
for all $t\ge0$ if $\theta$ is large enough.

For $h-\eta<x<h$ one has $\zeta\in(0,1)$ and
$\underline u=\psi(\zeta)$. Since $\eta'=\tilde\eps\eta/(t+\theta)$,
\begin{equation}\label{eq:half-ut}
\p_t\underline u
=\psi'(\zeta)\left[\frac{h'}{\eta}-\frac{\zeta\eta'}{\eta}\right]
=\psi'(\zeta)\left[\frac{h'}{\eta}-\frac{\zeta\tilde\eps}{t+\theta}\right]
\le\psi'(\zeta)\frac{h'}{\eta}.
\end{equation}
Using \eqref{eq:psi-der-bound} and $h'=K_1(\ln(t+\theta)+1)$, we obtain
\begin{equation}\label{eq:half-ut-bound}
\p_t\underline u
\le C_\psi\zeta^{\rho-1}\,\frac{K_1(\ln(t+\theta)+1)}{\eta}
=\frac{C_\psi K_1(\ln(t+\theta)+1)}{\eta}\zeta^{\rho-1}.
\end{equation}

Let $C_0,\delta_0$ be as in Lemma~\ref{lem:profile-half}, and define
\begin{equation}\label{eq:half-zeta}
C_3:=\left(\frac{dC_0}{2C_\psi K_1}\right)^{1/(\rho-1)},\qquad
\zeta_*(t,\theta):=C_3\bigl(\ln(t+\theta)+1\bigr)^{-1/(\rho-1)}.
\end{equation}
Since $\zeta_*(t,\theta)\to0$ as $\theta\to\infty$, for $\theta$ large we have
$\zeta_*(t,\theta)\le\delta_0$ for all $t\ge0$.

\smallskip

We now continue the discussion according to the location of $x$ and divide the range $h-\eta\le x< h$ into two subregions.

\underline{Region a:} $h-\zeta_*(t,\theta)\eta\le x<h$ (so
$\zeta\le\zeta_*(t,\theta)\le\delta_0$). By Lemma~\ref{lem:profile-half}(ii), for $x$ in this region,
$(-\Delta)^{1/2}\underline u\le -C_0\eta^{-1}$. Moreover, by
\eqref{eq:half-ut-bound} and \eqref{eq:half-zeta},
\[
\p_t\underline u
\le\frac{C_\psi K_1(\ln(t+\theta)+1)}{\eta}\,\zeta_*^{\rho-1}
=\frac{C_\psi K_1 C_3^{\rho-1}}{\eta}
=\frac{dC_0}{2\eta}
\le -\frac d2(-\Delta)^{1/2}\underline u.
\]
Since $f(\underline u)\ge0$, we thus obtain
\[
\p_t\underline u+d(-\Delta)^{1/2}\underline u
\le f(\underline u).
\]

\smallskip

\underline{Region b:} $h-\eta<x<h-\zeta_*(t,\theta)\eta$ (so
$\zeta_*(t,\theta)<\zeta<1$). In this region, $\underline u=\psi(\zeta)\ge c_\psi\zeta^\rho$,
and so by \eqref{eq:half-ut-bound} we obtain
\begin{equation}\label{eq:half-ut-over-u}
\frac{\p_t\underline u}{\underline u}
\le\frac{C_\psi K_1}{c_\psi}\,\zeta^{-1}\,\frac{\ln(t+\theta)+1}{\eta}
\le\frac{C_\psi K_1}{c_\psi C_3}\,
\bigl(\ln(t+\theta)+1\bigr)^{1+\frac1{\rho-1}}\eta^{-1}.
\end{equation}
 Moreover,
\[
\underline u\ge c_\psi\zeta^\rho\ge c_\psi\zeta_*^\rho=c_\psi C_3^\rho(\ln(t+\theta)+1)^{
-\rho/(\rho-1)}.
\]
 Since $\eta^{-1}=(t+\theta)^{-\tilde\eps}$ with
$\tilde\eps>0$, both
$\bigl(\ln(t+\theta)+1\bigr)^{1+\frac1{\rho-1}}\eta^{-1}$ and
$\eta^{-1}\bigl(\ln(t+\theta)+1\bigr)^{\rho/(\rho-1)}$ tend to $0$ uniformly for $t\geq 0$ as
$\theta\to\infty$. Hence, for $\theta$ large,
\begin{equation}\label{eq:half-region2b}
\frac{\p_t\underline u}{\underline u}\le\frac{\sigma_0}{2},\qquad
\frac{dC_1\eta^{-1}}{\underline u}\le \frac{dC_1}{c_\psi C_3^\rho}\eta^{-1}\bigl(\ln(t+\theta)+1\bigr)^{\rho/(\rho-1)}\leq \frac{\sigma_0}{2}.
\end{equation}
By Lemma~\ref{lem:profile-half}(i),
$|(-\Delta)^{1/2}\underline u|\le C_1\eta^{-1}$.
Consequently
\[
\p_t\underline u+d(-\Delta)^{1/2}\underline u
\le\frac{\sigma_0}{2}\underline u+dC_1\eta^{-1}
\le\frac{\sigma_0}{2}\underline u+\frac{\sigma_0}{2}\underline u
=\sigma_0\underline u\le f(\underline u).
\]

Thus \eqref{eq:half-interior} holds for $\theta$
large,  and so $(\underline u,\underline g,\underline h)$ is a lower solution to \eqref{eq 1.1}.

\smallskip

\textbf{Step 3:} We complete the proof by a comparison argument.

Since spreading occurs, there is $t_0>0$ such that
$$[\underline g(0),\underline h(0)]=[-K_1\theta\ln\theta,
K_1\theta\ln\theta]\subset[g(t_0),h(t_0)]$$ and
$u(t_0,x)\geq 1-\eps=K_2\geq \underline u(0,x)$ for
$x\in[\underline g(0),\underline h(0)]$. Hence we can apply the
comparison principle to obtain
\[
[\underline g(t),\underline h(t)]\subset[g(t_0+t),h(t_0+t)], \ \underline u(t,x)\leq u(t_0+t,x) 
\mbox{ for } t>0,\ x\in\R.
\]
Hence
\[
\liminf_{t\to\infty}\frac{h(t)}{t\ln t}=\liminf_{t\to\infty}\frac{h(t_0+t)}{t\ln t}\ge \lim_{t\to\infty} \frac{\underline h(t)}{t\ln t}=K_1=(1-\eps)(1-\tilde\eps)\mu.
\]
Letting $\eps,\tilde\eps\to0$ yields \eqref{eq:half-lower}. The estimate
for $-g(t)$ follows by symmetry.
\end{proof}


For the upper bound we will need the following  boundary barrier for the
fractional Laplacian.

\begin{lemma}\label{lem:barrier}
Suppose $s\in (0,1)$, ${\bf (f1)}$ holds, and $(u,g,h)$ is a solution of \eqref{eq 1.1}. Then there exists a constant $C_0>0$
such that for all $t> 1$ and $x\in[g(t),h(t)]$,
\begin{equation}\label{eq:barrier}
u(t,x)\le C_0\bigl(h(t)-x\bigr)^{s},
\qquad
u(t,x)\le C_0\bigl(x-g(t)\bigr)^{s}.
\end{equation}
\end{lemma}
\begin{proof}
Using ${\bf (f1)}$ we can compare \eqref{eq 1.1} with  the corresponding ODE problem to see that $\|u\|_\infty$ is finite.
Let $K:=\max_{u\in [0, \|u\|_\infty]}f(u)$ and
for any $\tau>1$ we now consider the problem
\begin{equation}\label{tau}
\begin{cases}
v_t+(-\Delta)^{s} v=K & \mbox{ for } t\in (0, 1],\ x\in (g(\tau), h(\tau)),\\
v=0 & \mbox{ for } t\in (0, 1], \  \ x\in\R\setminus (g(\tau), h(\tau)),\\
v(0,x)=u(\tau-1, x) & \mbox{ for }  x\in\R.
\end{cases}
\end{equation}
By Lemma 2.8 of \cite{CDN26}, there exists $C_0>0$ depending only on $\|u\|_\infty$ and $K$ such that
the unique solution $v(t,x)$ of \eqref{tau} satisfies
\[
v(t,x)\leq C_0 \min\{1, \frac{\rho^{s}(x)}{\sqrt{t}}\} \ \mbox{ for } t\in (0, 1],\ x\in (g(\tau), h(\tau)),
\]
where $\rho(x):={\rm dist} (x, \partial (g(\tau), h(\tau)))=\min\{h(\tau)-x, x-g(\tau)\}$ for $x\in (g(\tau), h(\tau))$.

Using the comparison principle we easily see that
\[
u(t+\tau-1,x)\leq v(t,x) \mbox{ for } t\in (0, 1],\ x\in\R.
\]
In particular,
\[
u(\tau, x)\leq v(1,x)\leq C_0\min\{1, \rho^{s}(x)\} \ \mbox{ for }  x\in (g(\tau), h(\tau)).
\]
Since this holds for any $\tau>1$, the desired estimates in \eqref{eq:barrier} follow immediately.
\end{proof}

\begin{lemma}\label{lem:half-upper}
Under the assumptions of Theorem~\ref{teo 1}, if $s=1/2$
then
\begin{equation}\label{eq:half-upper}
\limsup_{t\to\infty}\frac{h(t)}{t\ln t}\le\mu,\qquad
\limsup_{t\to\infty}\frac{-g(t)}{t\ln t}\le\mu.
\end{equation}
\end{lemma}

\begin{proof}
Fix $T_0>1$ such that
$u(t,x)\le1+\eps$ for all $t\geq T_0$ and  $x\in[g(t),h(t)]$. For $s=1/2$ the free boundary
condition reads
\[
h'(t)=\mu\int_{g(t)}^{h(t)}\frac{u(t,x)}{h(t)-x}\,\dd x.
\]
Splitting the integral, using $u\le1+\eps$ on $[g(t),h(t)-1]$ and 
\eqref{eq:barrier} on $[h(t)-1,h(t)]$, we obtain
\begin{equation}\label{h'}\begin{aligned}
h'(t)
&\le\mu(1+\eps)\int_{g(t)}^{h(t)-1}\frac{\dd x}{h(t)-x}
+\mu C_0\int_{h(t)-1}^{h(t)}\frac{\dd x}{(h(t)-x)^{1/2}}\\
&=\mu(1+\eps)\ln\bigl(h(t)-g(t)-1\bigr)+2\mu C_0.
\end{aligned}
\end{equation}
By symmetry (and using the second inequality in \eqref{eq:barrier}),
\begin{equation}\label{eq:half-ode}
-g'(t)\le\mu(1+\eps)\ln\bigl(h(t)-g(t)-1\bigr)+2\mu C_0.
\end{equation}

Set $M(t):=\max\{h(t),-g(t)\}$. Then
$h(t)-g(t)\le2M(t)$, and hence
$\ln(h(t)-g(t))\le\ln M(t)+\ln2$. We note that $M'(t)$ exists for a.e. $t>0$, and the left and right derivatives
 $M'(t-0)$ and $M'(t+0)$ exist for every $t>0$. In case they are not the same, for convenience of later use, we define $M'(t)=M'(t-0)$.
Then
$M'(t)\le\max\{h'(t),-g'(t)\}$ for  every $t>0$, and it follows from \eqref{h'} and \eqref{eq:half-ode} that, for all $t\ge T_0$,
\[
M'(t)\le\mu(1+\eps)\ln M(t)+\mu(1+\eps)\ln2+2\mu C_0.
\]
 Since $M(t)\to\infty$ as $t\to\infty$,  for any given $\delta>0$ and for all large $t$, 
we have 
\[
M'(t)\le\mu(1+\eps)(1+\delta)\ln M(t).
\]

 Let 
 \[\mbox{$z(t):=\sigma (1+\delta)\, t\ln t$ with $\sigma:=\mu(1+\eps)(1+\delta)$.}
 \]
 Then
 \[
 \ln z(t)=\ln t+\ln \ln t+\ln [\sigma(1+\delta)]< (1+\delta)\ln t \mbox{ for } t\gg1.
 \]
 Hence
 \[
 z'(t)=\sigma (1+\delta) (\ln t+1)>\sigma(1+\delta) \ln t> \sigma \ln z(t) \mbox{ for } t\gg 1.
 \]
 Fix $T_1\geq T_0$ large so that
 \[
   M(T_0)< z(T_1) \mbox{ and } \begin{cases}M'(t)\le \sigma \ln M(t) & \mbox{ for } t\geq T_0,\\
 z'(t)> \sigma \ln z(t)& \mbox{ for } t\geq T_1. \end{cases}
 \]
 We claim that $M(T_0+t)<z(T_1+t)$ for all $t>0$. Clearly this is true for $t>0$ small. If the claim is false, then there exists a first time moment $t_1>0$ such that
 \[
 M(T_0+t_1)=z(T_1+t_1),\ M(T_0+t)<z(T_1+t) \mbox{ for } t\in [0, t_1).
 \]
 Therefore $M'(T_0+t_1)\geq z'(T_1+t_1)$. On the other hand,
 \[
 M'(T_0+t_1)\leq \sigma \ln M(T_0+t_1) =\sigma \ln z(T_1+t_1)<z'(T_1+t_1).
 \]
 This contradiction proves the claim. Hence
 \[
M(T_0+t)<z(T_1+t) \mbox{ for } t>0,
\]
and
\[
\limsup_{t\to\infty}\frac{M(t)}{t\ln t}=\limsup_{t\to\infty}\frac{M(T_0+t)}{t\ln t}\le\lim_{t\to\infty}\frac{z(T_1+t)}{t\ln t}=\sigma (1+\delta).
\]
Since $h(t)\le M(t)$ and $-g(t)\le M(t)$, letting $\delta\to0$ and
$\eps\to0$ gives \eqref{eq:half-upper}.
\end{proof}

Clearly Theorem~\ref{teo 1}(ii) is a consequence of
 Lemmas~\ref{lem:half-lower} and \ref{lem:half-upper}. We note again that the lower bound for $u(t,x)$ in the proof of Lemma \ref{lem:half-lower} and the conclusion in Lemma \ref{lem:half-upper} imply that the boundary of the level set $\{x: u(t,x)>\sigma\}$ $(0<\sigma<1)$ propagates with the same rate as $h(t)$ and $g(t)$ when $t\to\infty$.


\begin{thebibliography}{1000} 



\bibitem{CR13}
X.~Cabr\'e, J.-M.~Roquejoffre,
{The influence of fractional diffusion in Fisher--KPP equations},
Comm.\ Math.\ Phys.,\ \textbf{320} (2013), 679--722.

\bibitem{CDLL}
J.-F.~Cao, Y.~Du, F.~Li, W.-T.~Li,
 The dynamics of a Fisher--KPP nonlocal diffusion model with free
boundaries,
J.\ Funct.\ Anal.,\ \textbf{277} (2019), 2772--2814.

\bibitem{CDN26} H. Chen, Y. Du, W. Ni, The fractional Fisher-KPP equation with free boundaries,
{Preprint}, 2026 (arXiv:2609.34504).

\bibitem{CKS}	Z.-Q. Chen, P. Kim and R. Song,  Heat kernel estimates for Dirichlet fractional Laplacian, 
{ J. Eur. Math. Soc.}, {\bf 12} (2010), 1307-1329.



\bibitem{DLZ}
Y.~Du, F.~Li, M.~Zhou,
{Semi-wave and spreading speed of the nonlocal Fisher--KPP equation
with free boundaries},
J.\ Math.\ Pures Appl.,\ \textbf{154} (2021), 30--66.


\bibitem{DL}
Y.~Du, Z.~Lin,
{Spreading--vanishing dichotomy in the diffusive logistic model with a
free boundary},
SIAM J.\ Math.\ Anal.,\ \textbf{42} (2010), 377--405.

\bibitem{DLou}
Y.~Du, B.~Lou,
{Spreading and vanishing in nonlinear diffusion problems with free
boundaries},
J.\ Eur.\ Math.\ Soc.,\ \textbf{17} (2015), 2673--2724.

	
	\bibitem{DMZ}
		Y. Du, H. Matsuzawa and M. Zhou, 
Sharp estimate of the spreading speed determined by nonlinear free boundary problems, 
 SIAM J. Math. Anal., {\bf 46} (2014), 375-396.

\bibitem{DN-MA}
Y.~Du, W.~Ni,
{Exact rate of accelerated propagation in the Fisher--KPP equation with
nonlocal diffusion and free boundaries},
Math.\ Ann.,\ \textbf{389} (2024), 179--216.


\bibitem {DN25}  Y. Du, W. Ni, Rate of propagation for the Fisher-KPP equation with nonlocal diffusion and free boundaries,
{ J. Eur. Math. Soc.,} {\bf  27} (2025)  1267--1319.


\bibitem {FY13}  P. Felmer,  M.  Yangari,  
Fast propagation for fractional KPP equations with slowly decaying initial conditions,
{SIAM J. Math. Anal.}, {\bf 45} (2013), no. 2, 662--678.

\bibitem{FR2017} X. Fern\'andez-Real and  X. Ros-Oton,
	Regularity theory for general stable operators: Parabolic equations,
	{J. Funct. Anal.}, {\bf 272} (2017), 4165-4221.
	
	\bibitem{F} R. A. Fisher, The wave of advance of advantageous genes, 
{Ann. Eugenics}, {\bf 7} (1937), 335-369.

	
\bibitem{GJZ} 
T. Grzywny, T. Jakubowski and G. Zurek, Green function for gradient perturbation of unimodal L\'{e}vy
processes in the real line, {Bull. Malays. Math. Sci. Soc.}, {\bf  43} (2020), 1223-1251.

	
\bibitem{KPP} A. N. Kolmogorov, I. G. Petrovsky and N. S. Piskunov, A study of the diffusion equation with increase in the amount of substance, and its application to a biological problem, { Bull. Moscow Univ. Math. Mech.}, {\bf 1} (1937), 1-25.
	
\bibitem {K21} T. Kukuljan, 
The fractional obstacle problem with drift: higher regularity of free boundaries, {J. Funct. Anal.}, {\bf 281} (2021), no. 8, Paper No. 109114, 74 pp.


\bibitem {S07} L. Silvestre, Regularity of the obstacle problem for a fractional power of the Laplace operator,
  {Comm. Pure Appl. Math.}, {\bf  60} (2007), 67-112.
 

 \end{thebibliography}
 \end{document}